\documentclass[12pt,a4paper]{article}
\usepackage[latin1]{inputenc}
\usepackage{amsmath}
\usepackage{amsfonts}
\usepackage{amssymb}
\usepackage{amsthm}
\usepackage{verbatim}
\usepackage{graphicx}
\usepackage[normalem]{ulem}
\numberwithin{equation}{section}
\newtheorem{lm}{Lemma}
\newtheorem{thrm}{Theorem}
\newtheorem{cor}{Corollary}
\newtheorem{defn}{Definition}

\newcommand{\domain}{\Omega}
\newcommand{\fb}{F\!B}
\newcommand{\ep}{\varepsilon}

\newcommand{\asj}{a^s_{ij}}
\newcommand{\pob}{\underline{P_0}}
\newcommand{\pof}{\overline{P_0}}

\begin{document}
\author{Thomas Backing\thanks{Partially Supported by NSF DMS-1101246 and the Purdue Research Foundation}\\ Purdue University}
\date{}

\title{Regularity of the Free Boundary for a Bernoulli-Type Parabolic Problem with Variable Coefficients}

\maketitle
\section{Introduction}

 In this paper we continue the study of a class of parabolic free boundary problems initiated in [B]. Our present goal is to establish that Lipschitz free boundaries are $C^{1,\alpha}$ surfaces in a sense that is stated precisely in Theorem~\ref{thrm:main_thrm}.

The problem under consideration in this work is as follows: Let $u$ be a viscosity solution in a domain $\Omega$ to
\begin{equation}\label{FBP_Statement}
\left\lbrace
\begin{aligned}
&\mathcal{L}u -u_t =0 \quad \text{in } \;\{u>0\}\cup\{u\leq 0\}^\circ\\
 &G(u^+_\nu,u^-_\nu)=1 \quad \text{along } \;\partial\{u>0\}.
\end{aligned}
\right.
\
\end{equation}
Here $\mathcal{L}$ is an elliptic operator with H\"{o}lder continuous coefficients and $G(\;,\;)$ defines the free boundary condition of the problem.

Typical examples of the boundary condition in~\eqref{FBP_Statement} include $u^+_\nu =1$ and $(u^+_\nu)^2-(u^-_\nu)^2 =2M$, $M$ a positive number. Both arise as the free boundary condition for a singular perturbation problem which models combustion. This problem consists of studying the limit as $\ep\rightarrow 0$ of solutions to
\[
\Delta u^\ep -u^\ep_t=\beta_\ep(u^\ep)
\]
where $\beta(s)$ is a Lipschitz function supported on $[0,1]$ with
\[
\int_0^1 \beta(s) \mathrm{d}s =M \quad \text{and} \quad \beta_\ep(s) =\frac{1}{\ep}\beta(\frac{s}{\ep}).
\]

Under the assumption that $u^\ep \geq 0$, it was shown in \cite{CV} that the boundary condition for the limit function $u$ is $u_\nu^+=1$ . In \cite{CLW1}, \cite{CLW2} the two phase version of this problem was studied and the free boundary condition for the limit solution was demonstrated to be $(u^+_\nu)^2 -(u^-_\nu)^2 =2M$. In both cases this free boundary condition holds in a suitable weak sense.

Stationary (i.e. time independent) versions of~\eqref{FBP_Statement} with $\mathcal{L}=\Delta$ were studied in \cite{C1}, \cite{C2}. In these pioneering papers the idea of a viscosity solution to~\eqref{FBP_Statement} was introduced and the key concepts of monotonicity cones and sup-convolutions were introduced. The main result of these works is that Lipschitz free boundaries are smooth, as are sufficiently `flat' free boundaries. In this context `flat' means that the free boundary is close to the graph of a Lipschitz function with suitably small Lipschitz constant.

The first extension of these techniques to a parabolic problem was in \cite{ACS1}, \cite{ACS2}, \cite{ACS3}. The problem studied in these works is the Stefan problem, which models melting/solidification and differs from~\eqref{FBP_Statement} in its free boundary condition which, among other differences, involves the time derivative of $u$. Similar, though not quite as strong, results were proved in these papers as for the elliptic problem studied in \cite{C1}, \cite{C2}. Lipschitz free boundaries were proved to be $C^1$ under a non-degeneracy condition on $u$, and sufficiently flat free boundaries were also proved to be $C^1$. In both cases it was also proved that $u\in C^1(\overline{\Omega^+})\cup C^1(\overline{\Omega^-})
$, so that $u$ satisfies the free boundary condition in a classical sense. Finally, \cite{F} adapted these techniques to the study of~\eqref{FBP_Statement} for the heat equation.

All of the above cited works on the regularity of the free boundary involve either the Laplacian in the stationary case or the heat equation in the parabolic case. The proofs in these papers make extensive use of the fact that directional derivatives of solutions to a constant coefficient linear PDE are themselves solutions to the same equation. Indeed, the most difficult aspect of adapting these methods to the variable coefficient case is that this fact is unavailable. The only progress in adapting these methods to problems with variable coefficients is found in \cite{CFS}\textbf{,} where the authors study an elliptic problem\textbf{,} and in \cite{FS1}, \cite{FS2} where the authors study the Stefan problem with flat free boundaries.

In this work we adapt these methods and use them to study the regularity of the free boundary to solutions of~\eqref{FBP_Statement}. Our main result is that the free boundary is a differentiable surface whose normal varies with a H\"{o}lder modulus of continuity and the the free boundary condition is taken up with continuity.

The outline of this work is as follows: In Section 2 we precisely define the problem, the concept of a solution, our assumptions and our main result. In Section 3 we have collected the main tools and known results that we will need in our analysis. Section 4 deals with the interior enlargement of the monotonicity cone while Section 5 contains results that propagate a portion of this enlargement to the free boundary. Finally Section 6 contains the iteration used to prove the regularity of the free boundary in space while Section 7 contains a similar iteration used to prove the regularity in space-time.

\section{Definitions and Statement of Results}
We collect in this section the precise statement and hypotheses of our problem along with the statement of our main result.

We will denote the positivity set of $u$ by $\domain^+$; likewise the negative set is denoted by $\domain^-$. Occasionally we will write $\domain^\pm(u)$ to emphasize the dependence of these domains on the function $u$. The set $\partial\{u>0\}$ is the free boundary and will be denoted by $\fb(u)$ or just $\fb$. In this work we will assume that the free boundary is the graph of a Lipschitz function $f$, that is, it consists of the set $\{(x',x_n,t)| f(x',t)=x_n\}$ with $f(0,0) =0$. Denote by $L$ and $L_0$ the Lipschitz constant of $f$ in space and time respectively.

 The operator in~\eqref{FBP_Statement}
\[
\mathcal{L} =\sum_{i,j} a_{ij}(x,t)D_{ij}
\] has H\"{o}lder continuous coefficients $a_{ij} \in C^{0,\alpha}(\domain)$ with respect to the parabolic distance, $0<\alpha\leq 1$ and there exists $\lambda, \Lambda >0$ such that
\[
\lambda|\xi|^2 \leq \sum a_{ij}(x,t)\xi_i \xi_j \leq \Lambda |\xi|^2
\] for all $(x,t) \in \domain$. Denoting by $A(x,t)$ the matrix $[a_{ij}(x,t)]$,  we assume $A(0,0) =[\delta_{ij}]$ the identity.

On $G(a,b)$ we will require
\begin{enumerate}
\item $G$  Lipschitz with constant $L_G$ in both variables.
\item $G(a_1,b) -G(a_2,b) >c^*(a_1-a_2)^p$ if $a_1>a_2$ (strictly increasing in first variable)
\item $G(a,b_1)-G(a,b_2) < -c^*(b_1-b_2)^p$ if $b_1>b_2$ (strictly decreasing in second variable)
\end{enumerate}
The $p$ appearing here is some positive power.
\begin{defn}(Classical Subsolution/Supersolution)
We say $v(x,t)$ is a classical subsolution (supersolution) to~\eqref{FBP_Statement} if $v \in C^1(\overline{\domain^+(v)})\cup C^1(\overline{\domain^-(v)})$, $\mathcal{L}v-v_t\geq 0$ $(\mathcal{L}v-v_t\leq 0)$ in $\domain^\pm(v)$ and
\[
G(v_\nu^+,v_\nu^-) \geq 1 \;\;(G(v_\nu^+,v_\nu^-) \leq 1), \quad \text{where }\nu= \frac{\nabla v^+}{|\nabla v^+|}\textbf{.}\]
A strict subsolution (supersolution) satisfies the above with strict inequalities.
\end{defn}

\begin{defn}(Viscosity Subsolutions/Supersolutions)
A continuous function $v(x,t)$ is a viscosity subsolution (supersolution) to~\eqref{FBP_Statement} in $\domain$ if for every space-time cylinder $Q= B_r' \times (-T,T) \Subset \domain$ and for every classical supersolution (subsolution) $w$ in $Q$, the inequality $v\leq w$ ($v\geq w)$ on $\partial_pQ$ implies that $v\leq w$ ($v\geq w)$. Additionally, if $w$ is a strict classical supersolution (subsolution), then $v<w$ $(v>w)$ on $\partial_p Q$ implies $v<w$ $(v>w)$ inside $Q$.
\end{defn}

We now turn to the hypotheses on the free boundary of $u$. The main theorem of this work will require that this free boundary is Lipschitz, but will we will also require a non-degeneracy condition to hold at regular points. We first define such points.

\begin{defn}(Regular Points) A point $(x_0,t_0)$ on the free boundary of $u$ is a right regular point if there exists a space-time ball $B_R \subset \domain^+$ such that $B_R \cap \partial \{u>0\} =\{(x_0,t_0)\}$.

 A point $(x_0,t_0)$ on the free boundary of $u$ is a left regular point if there exists a space-time ball $B_R \subset \domain^-$ such that $B_R \cap \partial \{u\leq 0\} =\{(x_0,t_0)\}$.

\end{defn}

We will assume the following non-degeneracy condition on $u$: There exists a $m>0$ such that if $(x_0,t_0)$ is a right regular point for $u$ then
\begin{equation}
\frac{1}{|B'_r(x_0)|}\int_{B'_r(x_0)} u^+ dx \geq mr.
\end{equation}

The main result of this paper is the following theorem.
\begin{thrm}\label{thrm:main_thrm}
Let $u$ be a solution to our free boundary problem in $Q_1$ satisfying the hypotheses of this section. Then for every point $(x,t)$ on the free boundary in $Q_{1/2}$ there exists a normal vector to the surface $\eta(x,t)$. Furthermore, this normal vector satisfies
\begin{enumerate}
\item $|\eta(x,t) -\eta(y,t)| \leq C|x-y|^\alpha$
\item $|\eta(x,s) -\eta(x,t)| \leq C|s-t|^\beta$
\end{enumerate}
Finally, the free boundary condition is taken up with continuity by the solution $u$ so that $u$ is a classical solution to~\eqref{FBP_Statement}.
\end{thrm}

\section{Main Tools and Collected Results}

Define the domain $\Omega_{2r}$ by
\[\Omega_{2r} =\{(x',x_n,t) : |x'|<2L^{-1}r, |t|<4L_0^{-2}r^2, f(x',t)<x_n<4r\}.
\]

Denote by $P_r =(0,r,0)$, $\overline{P_r}=(0,r,2L_0^{-2}r^2)$, $\underline{P_r}=(0,r,-2L_0^{-2}r^2)$. These are the inward point, forward point and backward point, respectively. Denote by $\delta(X,Y)$ the parabolic distance between $X=(x,t)$ and $Y=(y,s)$ and by $\delta_X$ the parabolic distance from $X$ to the origin.\\

Our tools, valid for $\mathcal{L}$-caloric functions on Lipschitz domains vanishing on a piece of the boundary, are as follows (see \cite{FS2}):

\textit{Interior Harnack Inequality}: There exists a positive constant $c=c(n,\lambda,\Lambda)$ such that for any $r\in(0,1)$
\[
u(\underline{P_r}) \leq cu(\overline{P_r}).
\]

\textit{Carleson Estimate}: There exists a $c=c(n,\lambda,\Lambda,L,L_0)$ and $\beta=\beta(n,\lambda,\Lambda,L,L_0)$, $0<\beta\leq 1$ such that for every $X\in \Omega_{r/2}$
\[
u(X) \leq c\left(\frac{\delta_X}{r}\right)^\beta u(\overline{P_r}).
\]

\textit{Boundary Harnack Principle}: There exists $c=c(n,\lambda, \Lambda, L, L_0)$ and $\beta=\beta(n,\lambda, \Lambda, L, L_0)$, $0<\beta \leq 1$, such that for every $(x,t) \in \Omega_{2r}$ and $u$ and $v $ are two solutions
\[
\frac{u(x,t)}{v(x,t)} \geq c \frac{u(\underline{P}_r)}{v(\overline{P}_r)}\textbf{.}
\]

\textit{Backward Harnack Inequality}: Let $m=u(\underline{P_{3/2}})$ and $M=\sup_{\Omega_2}u$. Then there exists a positive constant $c=c(n,\lambda,\Lambda,L,L_0,M/m)$ such that if $r\leq 1/2$
\[
u(\overline{P_r}) \leq cu(\underline{P_r}).
\]

Throughout the work we will use $c$ to denote constants which depend on some or all of $n,\lambda,\Lambda,L,L_0,M/m$. 

Our starting point in the analysis of the free boundary will be the following result proved in \cite{B}. We denote the cone of directions with opening $\theta$ and axis $\eta$ by $\Gamma(\theta,\eta)$.

\begin{thrm}
Let $u$ be a viscosity solution to~\eqref{FBP_Statement} satisfying the hypotheses of this section. Then $u$ is Lipschitz and possesses a space-time cone of directions with axis $e_n$ and opening angle $\theta$ in which the solution is monotone:
\[
u(x-\tau)\leq u(x) \quad \forall \tau \in \Gamma(\theta,e_n)
\]

\end{thrm}

\subsection{Initial Configurations and Domains}

In what follows it will be necessary to know that the coefficients $a_{ij}(x,t)$ in the operator $\mathcal{L}$ are suitably close to $\delta_{ij}$. To this end we define $\asj(x,t) =a_{ij}(sx,s^2t)$ and set
\[
\mathcal{L}^s-\partial_t = \sum_{i,j} \asj(x,t) D_{ij} -\partial_t
\] to be the parabolic operator with these dilated coefficients. Set
\[
u_s(x,t) = \frac{u(sx,s^2t)}{s}.
\]
Then we have the equivalence
\[
\mathcal{L}u-u_t=0 \Leftrightarrow  \mathcal{L}^su_s-(u_s)t=0.
\] Note that this parabolic rescaling of $u$ does not alter the free boundary condition in~\eqref{FBP_Statement}. We will assume $a_{ij}(0,0) =\delta_{ij}$ and set $A=\sup_{i,j} [a_{ij}]_\alpha$.

Let $x_0=\frac{3}{4}e_n$, $P_0=(x_0,0)$, $\overline{P_0}=(x_0,\frac{9}{8L_0^2})$, $\underline{P_0} =(x_0,-\frac{9}{8L_0^2})$. These are inward, forward and backward reference points, respectively.

Next define regions $T= B'_{1/4}(x_0)\times (-\frac{9}{16L_0^2},\frac{9}{16L_0^2})$ and set
\[
\Phi = B'_{1/4-\sigma}(x_0)\times (-\frac{9}{16L_0^2}+\sigma^2,\frac{9}{16L_0^2}-\sigma^2)
\]
$\sigma$ to be specified later. By construction the parabolic distance from $\partial \Phi$ to $\partial T$ is $\sigma$.

Finally, let
 \[
 \Psi =B'_{1/8}(x_0)\times (-\frac{9}{32L_0^2},\frac{9}{32L_0^2}).
 \]
  In what follows we will have $\Psi \Subset \Phi \Subset T$ by our choice of $\sigma$. Additionally, by an initial change of variables $u(rx,rt)$, with $r<1$, we can reduce the Lipschitz constants $L$ and $L_0$ to be less than one, so that the above regions and test points are contained within $\Omega_4$. Finally, by a rescaling we have the free boundary of $u$ contained in $\{|x_n|<1/10\}$. \\

Next define $z$ by
\begin{align*}
\Delta z -z_t &=0 \quad \text{in } T\\
z &=u \quad \text{on } \partial_p T.
\end{align*}
Note that
\begin{align*}
\mathcal{L}^s(u-z)-(u-z)_t &= \sum (\asj-\delta_{ij})D_{ij}z \\
\Delta(u-z)-(u-z)_t &= \sum (\asj-\delta_{ij})D_{ij}u.
\end{align*}

We may assume by this configuration that the conclusion of Lemma 2.1 in [FS1] holds throughout $\Omega_4$. This states that
\[
c_1 \frac{u(X)}{d_{X}} \leq D_nu(X) \leq c_2 \frac{u(X)}{d_X}
\]
where here $d_X$ denotes the distance from $X=(x,t)$ to the FB at time level $t$.

\section{Interior Enlargement of the Monotonicity Cone}
 The results in this section only require the following: $u$ is $\mathcal{L}^s$-caloric, where $\mathcal{L}^s$ is suitably close to $\Delta$ (as controlled by the $\asj$),  $u$ vanishes on the piece of the boundary $\{f(x',t)=x_n\}$, and $u$ is Lipschitz with a monotonicity cone $\Gamma(e_n,\theta)$. In particular, the free boundary condition $G(u_\nu^+, u_\nu^-)=1$ plays no role in these results. Our method of proof is similar to \cite{CFS} in the elliptic case.

\begin{lm}
Let $u$ be a solution to $\mathcal{L}^s-u_t=0$ in $\Omega_4$,  $z$ as above. Then
\begin{equation}
|u-z|^*_{2+\alpha,\Phi} \leq Ks^\beta u(\pob)
\end{equation}
where $K=K(A)$ is a constant which depends on $A$ as well as the usual quantities and $\beta =\frac{\alpha^2}{\alpha+2}$.
\end{lm}

\begin{proof}

 We apply the Schauder estimates to the difference $u-z$ to obtain
\begin{equation}\label{inq:SchEst}
|u-z|^*_{2+\alpha,\Phi} \leq C(|u-z|_{0,\Phi}+|\sum (a_{ij}-\delta_{ij})D_{ij}u|^{(2)}_{0,\alpha, \Phi})
\end{equation}
using the standard (see [L]) notation for these norms and weighted norms. Recall that $|f|^{(2)}_\alpha = |f|^{(2)}_0+[f]^{(2)}_\alpha$\textbf{.} We begin by estimating the H\"{o}lder norm term as follows:
\begin{align*}
&|(\asj-\delta_{ij})D_{ij} u|^{(2)}_{0,\Phi} +[(\asj-\delta_{ij})D_{ij}u]^{(2)}_{\alpha,\Phi} \\
&\leq As^\alpha|D_{ij}u|^{(2)}_{0,\Phi} + |(\asj-\delta_{ij})|_{0,\Phi}[D_{ij}u]^{(2)}_{\alpha,\Phi}+[(\asj-\delta_{ij})]_{\alpha,\Phi}|D_{ij}u|^{(2)}_{0,\Phi}\\
&\leq As^\alpha|D_{ij}u|^{(2)}_{0,\Phi} + As^\alpha[D_{ij}u]^{(2)}_{\alpha,\Phi} + As^\alpha |D_{ij}u|^{(2)}_{0,\Phi}\\
&\leq A s^\alpha(|u|^*_{2+\alpha,\Phi}) \leq CAs^\alpha|u|_{0,\Phi} \leq CAs^\alpha u(\pof)\\
&\leq CAs^\alpha u(\pob).
\end{align*}

The Backward Harnack Inequality was used to obtain the last inequality.
Now we estimate the sup norm term in~\eqref{inq:SchEst}. Using the $\textit{a priori}$ estimates we have
\[
|u-z|_{0,\Phi} \leq |u-z|_{0,\partial_p\Phi}+ C'\sup_\Phi |(\asj-\delta_{ij})D_{ij}u|.
\]

The first term is estimated as follows: Recall that $u=z$ on the boundary of T so their difference is zero. Now $\partial_p \Psi$ lies a distance $\sigma$ from $\partial T$, so using the H\"{o}lder continuity up to the boundary we have that $|u-z|_{0,\partial_p \Psi} \leq c\sigma^\alpha|u-z|_{0,T}$.

Using this we have
\begin{align*}
|u-z|_{0,\Phi} &\leq c\sigma^\alpha|u-z|_{0,T} +C's^\alpha A \sigma^{-2} |u|_{0,\Phi}\\
&\leq c\sigma^\alpha|u|_{0,T} +C's^\alpha A \sigma^{-2} |u|_{0,\Phi}.
\end{align*}

Select $\sigma=s^{\frac{\alpha}{\alpha+2}}$ and obtain
\[
|u-z|_{0,\Phi} \leq s^{\frac{\alpha^2}{\alpha+2}}|u|_{0,T}(c+C'A) \leq c's^{\frac{\alpha^2}{\alpha+2}}u(\pof)(c+C'A).
\]
Now we always have $\alpha>\frac{\alpha^2}{\alpha+2}$ for $\alpha >0$ so for $s<1$ we have $s^\alpha< s^{\frac{\alpha^2}{\alpha+2}}$. Combining this with the estimate for the H\"{o}lder norm above we obtain
\[
|u-z|^*_{2+\alpha,\Phi} \leq \left[ CA +c'(c+C'A) \right]s^{\frac{\alpha^2}{\alpha+2}} u(\pof).
\]
This is the conclusion of the lemma with $K =(CA+c'(c+C'A))$.
\end{proof}

At this point it becomes convienent to begin treating the spacial portion of the cone and space-time cone seperately. We denote these by $\Gamma_x(e_n,\theta_x)$ and $\Gamma_t(\eta,\theta_t)$ respectively, $\eta$ a vector in the $e_n-e_t$ plane. We now focus on expanding these cones of directions.

\begin{lm}\label{lm:2}
Let $u$ be a solution to $\mathcal{L}^s-u_t=0$ in $\Omega_4$ with a cone of monotonicity $\Gamma(e_n,\theta)$. Let $\nabla =\frac{1}{|\nabla u(\pob)|}\nabla u(\pob)$. Then if $s$ is sufficiently small, for any $\tau \in \Gamma_x(e_n,\theta)$, $|\tau| =1$,
\begin{equation}\label{inq:HarnackLike}
D_\tau u(X) \geq (C\langle \nabla,\tau \rangle -cs^\beta)u(\pob)
\end{equation} for all $X\in \Psi$.
The same statement holds for $\tau \in \Gamma_t(\eta,\theta_t)$ with $\nabla$ the unit vector in direction $(u_{x_n}(\pob),u_t(\pob))$ in the $e_n-e_t$ plane.
\end{lm}

\begin{proof}
The proof follows the same lines in both the spacial and space-time cases. We begin with the spacial case.

Let $z$ be as in the previous lemma so that
\[
|u-z|_{2+\alpha,\Phi}^* \leq cs^\beta u(\pob)
\] where $\beta =\frac{\alpha^2}{\alpha+2}$.\\
Then $D_\tau z +cs^\beta u(\pob)$ is a non-negative solution to the heat equation in $\Phi$. By the Harnack Inequality, for $(x,t)$ in $\Psi$ we have
\begin{equation}
D_\tau z(x,t) +cs^\beta u(\pob) \geq c' \left( D_\tau z(\pob) +cs^\beta u(\pob)\right).
\end{equation}
Hence, letting $\nabla' =\frac{1}{|\nabla z(\pob)|}\nabla z(\pob)$, and assuming without loss of generality that $c'<1$, we obtain
\begin{align}
D_\tau z &\geq c'D_\tau z(\pob) -cs^\beta u(\pob)\\
&= c'|\nabla z(\pob)| \langle \nabla ',\tau \rangle -cs^\beta u(\pob).
\end{align}

Now using the Schauder estimate,  $\frac{u}{d}\sim |\nabla u|$ and the Harnack inequality we have
\[
|\nabla z| \geq |\nabla u| -cs^\beta u(\pob) \geq (C-cs^\beta)u(\pob).
\]
So if $s$ is small enough then $|\nabla z(\pob)| \geq cu(\pob)$ (4.6) becomes
\begin{equation} \label{eq:11}
D_\tau z \geq (c^*\langle \nabla',\tau \rangle -cs^\beta)u(\pob).
\end{equation}

Using the Schauder estimate once more we have
\[
|\nabla' - \nabla| \leq \frac{|\nabla u(\pob)-\nabla z(\pob)|}{|\nabla u(\pob)|}+ \frac{|(|\nabla z(\pob)|-|\nabla u(\pob)|)|}{|\nabla u(\pob)|} \leq cs^\beta.
\]
This is proved as follows: After adding and subtracting the quantity $|\nabla z|\nabla z$, the left-hand side becomes (suppressing the dependence on the point)
\[
\frac{|\nabla z |\nabla u|-\nabla u|\nabla z||}{|\nabla z||\nabla u|} = \frac{||\nabla z|(\nabla z -\nabla u)+\nabla z(|\nabla u|-|\nabla z|)|}{|\nabla z||\nabla u|}\mathbf{.}
\]
Applying the triangle inequality and canceling terms we obtain the first inequality. The second one is then a consequence of the Schauder estimate (with a different constant $c$). \\

We have thus established $\langle \nabla',\tau \rangle \geq \langle \nabla,\tau \rangle -cs^\beta$. Replacing this in~\eqref{eq:11} we have
\[
D_\tau z \geq (c_1\langle \nabla,\tau \rangle -cs^\beta)u(\pob).
\]
Finally, using the Schauder estimate one last time we obtain (with different constants than in the previous line)
\[
D_\tau u \geq (C\langle \nabla,\tau \rangle -cs^\beta)u(\pob).
\]

In the space-time case the same calculation works with $\nabla$ the unit vector in $e_n-e_t$ plane in direction $(u_{x_n}(\pob),u_t(\pob))$ and $\nabla'$ the vector in direction $(z_{x_n}(\pob),z_t(\pob))$ .
\end{proof}
\textbf{Remark:} In the case of the heat equation the inequality~\eqref{inq:HarnackLike}, without the `error' term $cs^\beta u(\pob)$, can be obtained easily by simply applying the Harnack Inequalities to the solution. In the variable coefficient case~\eqref{inq:HarnackLike} acts as a substitute.\\

 At this point we need to make sure that $D_\tau u$ remains positive, which cannot be guaranteed because of the error term in~\eqref{inq:HarnackLike}. To deal with this we eliminate the portion of the original cone consisting of the vectors which make an angle of more than $\frac{99\pi}{200}$ with $\nabla$. We denote this modified set of directions with $\Gamma'_x(e_n,\theta_x)$ (or $\Gamma_t'(\eta,\theta_t)$ as the case may be). Then for some $c_3$ and any $\tau \in \Gamma'_x(e_n,\theta_x)$
\[
\langle \nabla,\tau \rangle \geq c_3\delta\mathbf{,}
\] where $\delta =\frac{\pi}{2}-\theta_x$ is the defect angle of the cone, $c_3$ depends on how much of the cone was deleted. In the space-time case we use $\mu = \frac{\pi}{2}-\theta_t$ to denote the defect angle; initially this is the same as $\delta$ but this will not hold in the iteration later in the paper.

 As it is by now standard, this monotonicity can be described in terms of the sup-convolution, in our case over `thin' balls either purely spacial or in the space-time plane. Precisely,
\[
v_\varepsilon(X) = \sup_{B'_\varepsilon(X)} u(Y-\tau) \leq u(X)
\] for any $\tau \in \Gamma'(e_n,\frac{\theta}{2})$ sufficiently small, with $\varepsilon =|\tau|\sin\frac{\theta}{2}$. The $B'$ denotes a thin ball either purely in space or in space-time, depending on whether $\tau$ is in $\Gamma_x'$ or $\Gamma_t'$.\\

In what follows, the direction $\tau$ is either in $\Gamma'_x$ or $\Gamma_t'$; the proofs are the same. We distinguish between them only later work will make a distinction and it is convenient to have interior enlargement respect this distinction.
\begin{lm}\label{lm:3}
Let $u$ be as in Lemma~\ref{lm:2}. Then there exists $s_0>0$ such that if $s\leq s_0$ we have
\begin{equation}\label{inq:gap}
u(\pob) -v_\varepsilon(\pob) \geq \sigma\varepsilon u(\pob).
\end{equation}

\end{lm}

\begin{proof}
 If $Y\in B_\varepsilon(\pob)$ then, invoking the Mean Value Theorem with $\bar{\tau}=\tau+(\pob-Y)$, we obtain
\begin{equation}
u(Y-\tau) =u(\pob -\bar{\tau}) =u(\pob) -|\bar{\tau}|D_{\bar{\tau}}u(X^*).
\end{equation}
We estimate $D_{\bar{\tau}}u$ from below. If $\tau \in \Gamma(e_n,\frac{\theta}{2})$ then $\bar{\tau}\in \Gamma'(e_n,\theta)$, so using the observation immediately preceding this lemma we have
\begin{align*}
D_{\bar{\tau}} u &\geq (c_1\langle \nabla,\bar{\tau} \rangle -c_2s^\beta)u(\pob)\\
& \geq c\delta u(\pob)
\end{align*}
for any $s\leq s_0 =\left(\frac{c_1c_3\delta}{2c_2}\right)^{1/\beta}$\textbf{.}\\

 This, together with the fact $\bar{\tau}\geq c\varepsilon$, implies that $|\bar{\tau}|D_{\bar{\tau}}u(X^*) \geq c\varepsilon\delta u(\pob)$. Using this in (4.6) we get
\[
u(Y-\tau) \leq (1-c\varepsilon\delta) u(\pob)\textbf{.}
\]
Since $y$ is any point in $B_\varepsilon(\pob)$, we obtain the
the desired `gap' with $\sigma =c\delta$, which is the desired conclusion.
\end{proof}

 We now propagate the `gap' at the point $\pob$ in the above inequality to a smaller gap in a whole neighborhood.

\begin{lm}\label{lm:4}
Let $u$ be as Lemma~\ref{lm:3}, monotone increasing in every direction in $\Gamma'(e_n,\theta)$. Suppose for $\varepsilon>0$, $\sigma>0$ small we have
\begin{equation}\label{eq:lm4}
u(\pob) -v_\varepsilon(\pob) \geq \sigma\varepsilon u(\pob)\mathbf{.}
\end{equation}
Then there exists positive constants $C$ and $h$ such that in $\Psi$ we have
\[
u(X) -v_{(1+h\sigma)\varepsilon}(X) \geq C\sigma \varepsilon u(\pob)\mathbf{.}
\]
\end{lm}

\begin{proof}
Write $v_\varepsilon (X) =\sup_{B'_\varepsilon (X)} u_1$ where $u_1(X)=u(X-\tau)$.

Let $\tau \in \Gamma'(e_n,\theta/2)$\textbf{,} with $\varepsilon=|\tau|\sin \frac{\theta}{2}$. For any unit vector $\nu$ (either in space or in $e_n-e_t$ plane depending on whether $\tau \in \Gamma_x$ or $\Gamma_t$) write
\begin{align*}
u(P) -&u_1(P +\varepsilon\nu(1+h\sigma)) \\ &=[u(P)-u_1(P+\varepsilon\nu)]+[u_1(P+\varepsilon\nu) -u_1(P +\varepsilon\nu(1+h\sigma))]\\
&=W(P)+Y(P)\mathbf{.}
\end{align*}
Set $\bar{\tau} =\tau -\varepsilon\nu$. Then $|\bar{\tau}| \geq |\tau|-\varepsilon \geq c\varepsilon$. We estimate $W(P)$ and $Y(P)$ as follows:

$W(P)$ is non-negative (since $\bar{\tau}\in \Gamma(e_n,\theta)$) and a solution to a parabolic equation, hence we can apply the Harnack and conclude
\[
W(P)\geq cW(\pob) \geq c\sigma\varepsilon u(\pob)
\] using our initial assumption~\eqref{eq:lm4}.

For the $Y(P)$ term we apply the fact that $\nabla u \sim \frac{u}{d}$ and the Carleson Estimate. Hence
\[
 |\nabla u_1(P)| \leq Cu_1(P)  \leq Cu_1(\pob) \leq Cu(\pob)\mathbf{.}
\] Here we have used a combination of the Carleson Estimate and the Backward Harnack Inequality to obtain the middle inequality.

Together the estimates for $W(P)$ and $Y(P)$ yield
\[
W(P)+Y(P) \geq c\sigma\varepsilon u(\pob) -Ch\sigma\varepsilon u(\pob) \geq \bar{C}\sigma\varepsilon u(\pob)
\] if $h$ is chosen small enough ($h<\frac{c}{2C}$).

\end{proof}
 Using the Backward Harnack Inequality we have the following corollary.
\begin{cor}
Let $u$ be as in Lemma~\ref{lm:4}, monotone increasing in every direction in $\Gamma'(e_n,\theta)$. Suppose for $\varepsilon>0$, $\sigma>0$ small we have
\[
u(\pob) -v_\varepsilon(\pob) \geq \sigma\varepsilon u(\pob)\mathbf{.}
\]
Then there exists positive constants $C$ and $h$ such that in $\Psi$ we have
\[
u(X) -v_{(1+h\sigma)\varepsilon}(X) \geq C\sigma \varepsilon u(P_0)\mathbf{.}
\]
\end{cor}
 An application of the geometric cone enlargement lemma due to Caffarelli, Theorem 4.2 in [CS] yields an expansion of the monotonicity cone, either $\Gamma_t$ or $\Gamma_x$, as the case may be. This is stated precisely below.

\begin{cor}\label{cor:int_gain}
Let $u$ be a solution to our free boundary problem
\[
\left\lbrace
\begin{aligned}
&\mathcal{L}u -u_t =0 \quad \text{in } \;\{u>0\}\cup\{u<0\}\\
 &G(u^+_\nu,u^-_\nu)=1 \quad \text{along } \;\partial\{u>0\}
\end{aligned}
\right.
\
\]
and set $u_r = \frac{u(rx,r^2t)}{r}$ a parabolic blow-up. Then there exists an $r_0$ such that if $r\leq r_0$ we have the following:
\begin{enumerate}
\item If $u_r$ is monotone in a spacial cone of directions $\Gamma_0^x(e_n, \theta^x_0)$ then in $\Psi$ $u_r$ is monotone in an expanded cone of spacial directions $\Gamma_1^x(\nu_1,\theta_1^x)$ with defect angle decay given by $\delta_1\leq c\delta_0$ and $|\nu_1-e_n|\leq C\delta_0$.

\item If $u_r$ is monotone in a space-time cone of direction $\Gamma_0^t(\eta_0,\theta_0^t)$, $\eta_0 \in \mathrm{span}\{e_n,e_t\}$, then in $\Psi$ it is monotone in an expanded cone of directions $\Gamma_1^t(\eta_1,\theta_1^t)$ with defect angle decay $\mu_1 \leq c\mu_0$, $\eta_1 \in \mathrm{span}\{e_n,e_t\}$, $|\eta_1-\eta_0| \leq c\mu_0$.
\end{enumerate}
\end{cor}

\section{Propagation Lemma}

It is not possible to propagate the uniform gain in the monotonicity cone proved in the previous section to the free boundary. Instead only a portion of the gain can be propagated. This is accomplished by using a family of sup-convolutions with a variable radius.

 For a positive function $\phi$ and direction $\tau$ define the sup-convolution
\[
v_{\phi,\tau}(p) =\sup_{B_{\phi(p)}}u(q-\tau).
\]
Also, let $\mathcal{C}_{R,T} =B_R^\prime\times (-T,T)\mathbf{.}$

    In the sequel, we will need suitable versions of  Lemmas 3.1 \& 3.3 in [FS2]. The first describes a condition that $\phi$ needs to satisfy in order to make $v_{\phi,\tau}$ a sub/super-solution to our operator. The second establishes the existence of a family of functions satisfying this condition, among others.
 \begin{lm}
 Let $u$ be a solution to our free boundary problem for the operator $\mathcal{L}-D_t$. Let $\varepsilon_0$ be small enough and $\phi \in C^2(\bar{\mathcal{C}}_{R,T})$ be a strictly positive function. Let $\omega \leq \omega(\phi_{\mathrm{MAX}})$. Assume that in a smaller cylinder $\mathcal{C}^\prime \subset \mathcal{C}_{R,T}$ with $\mathrm{dist}(\mathcal{C}^\prime, \partial \mathcal{C}_{R,T}) \geq \rho \gg \varepsilon_0$ $D_t\phi\geq 0$ and
 \[
\mathcal{L}(\phi) -c_1 D_t\phi \geq C\frac{|\nabla \phi|^2 +\omega^2}{\phi} +c_2(|\nabla \phi |+\omega)
\] for some positive constants $C_0$, $C$, $c_1$ and $c_2$ depending only on $n,\lambda,\Lambda, \rho$.

Then in both $\Omega^\pm(v_{\psi,\tau}) \cap \mathcal{C}'$, $v_{\phi,\tau}$ is a viscosity subsolution to the operator $\mathcal{L} -D_t$.
 \end{lm}

 \textbf{Remark:} In [FS2] this lemma is stated for a family of operators and is therefore slightly different, and more complex, than the version we have stated. We do not require this in our case. Additionally, in [FS2], the lemma is stated for a solution $u$ to the Stefan problem, but the free boundary condition does not play a role in the proof and therefore the same result holds for our problem. Finally, their lemma has the Pucci extremal operator $\mathcal{P}^-$ on the left of the inequality instead of $\mathcal{L}$. Since $\mathcal{P}^-(\phi)-c_1D_t\phi \leq \mathcal{L}(\phi) -c_1D_t\phi$ by the properties of the Pucci operator, we are justified in making this substitution.\\

 Next define the region
 \[
 D=\left[ B_1'\setminus \left( \bar{B}'_{1/8}(x_0)\right)    \right]\times (-T,T).
 \]
From Lemma 3.3 \cite{FS2} we have the following:

 \begin{lm} Let $T>0$ and $C>1$. There exists positive constants $\bar{C}= \bar{C}(T,C)$, $k=k(T,C)$, and $h'_0 =h_0(T,C)$ such that for any $0<h'<h'_0$ there is a family of $C^2$ functions $\phi_\eta$, $0\leq \eta \leq 1$, defined in the closure of $D$ such that
 \begin{enumerate}
 \item $1-\omega\leq \phi_\eta \leq 1+\eta h'$
 \item $\mathcal{L}\phi_\eta-c_1D_t\phi_\eta -C\frac{|\nabla \phi|^2+\omega^2}{\phi_\eta} -c_2(|\nabla \phi|+\omega) \geq 0 $ in $D$
 \item $\phi_\eta\geq 1+k\eta h'$ in $B'_{1/2}\times(\frac{-T}{2},\frac{T}{2})$
 \item $\phi_\eta \leq 1$ in $D\setminus (B'_{7/8}\times (-\frac{7T}{8},\frac{7T}{8}))$
 \item $D_t\phi_\eta,  \leq \bar{C}\eta h'$ and $|\nabla \phi_\eta| \leq \bar{C}(\eta h'+\omega)$ in $\bar{D}$
 \item $D_t\phi_\eta \geq 0$ in $D$
 \end{enumerate}
 Here the $\omega$ appearing in (2) is a small positive constant; that is to say that if $c_1$, $c_2$, and $\omega$ are small positive constants depending on $n,\lambda,\Lambda,C$ then it is possible to construct this family.
 \end{lm}

\textbf{Remark:} This is essentially Lemma 3.3 in [FS2] with only two small differences. First, as noted above, our domain has only the one hole; this causes only small and obvious alterations to the construction. Second, similar to the previous lemma, Lemma 3.3 in [FS2] has the Pucci extremal operator $\mathcal{P}_-$ instead of $\mathcal{L}$ in item (2). As in the previous remark, we are justified in this substitution by the properties of $\mathcal{P}^-$.

This concludes the essential properties of the variable radii functions $\phi_\eta$. \\

In what follows we will use the family $\varepsilon \phi_{\sigma \eta}$. The $\sigma \eta$ term presents no difficulties but the derivative inequality which must be satisfied in order for the sup-convolutions to be subsolutions is not homogeneous in $\phi$. Precisely, when we replace $\phi$ with $\varepsilon\phi$ in item (2) of the above lemma we have
\[
\mathcal{L}\varepsilon\phi_\eta-c_1D_t\varepsilon\phi_\eta -C\frac{|\nabla \varepsilon\phi|^2+\omega^2}{\varepsilon\phi_\eta} -c_2(|\nabla \varepsilon\phi|+\omega)\mathbf{.}
\]
The presence of the $\omega$ terms prevents us from simply factoring out an $\varepsilon$. Rearranging we have
\[
\left(\mathcal{L}\varepsilon\phi_\eta-c_1D_t\varepsilon\phi_\eta -C\frac{|\nabla \varepsilon\phi_\eta|^2}{\varepsilon\phi_\eta}  -c_2|\nabla \varepsilon\phi_\eta|\right) -C\frac{\omega^2}{\varepsilon\phi_\eta}-c_2\omega\mathbf{.}
\]
Owing to the condition initially satisfied by $\phi_\eta$, the term in parentheses will be strictly positive provided $\omega$ is strictly positive (which it will be for a variable coefficient problem). So if we place the additional condition $\omega\leq \varepsilon^2$ then for sufficently small $\varepsilon$ the family $\varepsilon\phi_{\sigma\eta}$ will satisfy the desired inequality.\\

This condition is what restricts us to using $\varepsilon$-monotonicity. In our later work we take $\varepsilon =|\tau|\sin \delta$, $\tau$ and $\delta$ coming from the monotonicity cone. Since we are also requiring $\omega \leq \varepsilon^2$ we see that some restriction on the length of $\tau$ is necessary. We cannot take $\tau$ to be arbitrarily small since that would in turn force the oscillation of of coefficient matrix, which is measured by $\omega$, to be zero reducing the problem to the constant coefficient case.\\

We now prove our version of the propagation lemma used in this problems. From now on we will assume that $\omega_0 \leq \varepsilon^2$ with $\varepsilon\leq \varepsilon_0$. We will make use of standard asymptotic development results for both $u$ and $v_\ep$ as in \cite{B}.

\newcommand{\veta}{v_\eta}
\newcommand{\vetab}{\bar{v}_\eta}

\begin{lm}\label{lm:propagation}
Let $u_1$ and $u_2$ be two viscosity solutions to our problem in $B'_2\times (-2,2)$ and $F(u_2)$ Lipschitz continuous with $(0,0)\in F(u_2)$. Assume
\begin{enumerate}
\item In $B'_1\times(-T,T)$
\[
v_\varepsilon (x,t) = \sup_{B_\varepsilon(x,t)}u_1 \leq u_2(x,t)
\]
\item For some $\sigma$ positive and some $h$ small and $(x,t)\in B_{1/8}(x_0)\!\times\! (-T,T) \subset {\{u_2>0\}}$
\[
u_2(x,t) -v_{(1+h\sigma)\varepsilon}(x,t) \geq C\sigma\varepsilon u_2(x_0,0)
\]
\item $\omega_0$ is sufficiently small (as above).
\end{enumerate}
Then if $\varepsilon>0$ and $h>0$ are small enough, there exists $k\in (0,1)$ such that in $B_{1/2}\times (-\frac{T}{2},\frac{T}{2})$
\[
v_{(1+kh\sigma)\varepsilon}(x,t)\leq u_2(x,t)
\]
\end{lm}

 \begin{proof}
 Define $w(x,t)$ as follows:
\begin{align*}
\mathcal{L}w-w_t &=0 \quad \text{in } D\cap \{u_2>0\}\\
w &=u(P_0) \quad \text{on } \partial B_{1/8}(P_0)\times (-9T/10,9T/10)\\
w &=0 \quad \text{on the rest of } \partial_pD.
\end{align*}
Next, using the family constructed above with $\varepsilon\leq \varepsilon_0$, set
 \begin{align*}
 \veta &= v_{\varepsilon \phi_{\sigma \eta}}\\
 \vetab &= \veta +c\sigma \varepsilon w(x,t).
 \end{align*}

The constant $c$ is chosen to make $\vetab \leq u_2$ on $\partial_p [B'_{1/8}(P_0)\times(-9T/10,9T/10)]$. This is possible by the second hypothesis and the Harnack inequality. This ensures that $\bar{v}_0 \leq u_2$.

  We now demonstrate that the set of $\eta$ for which $\vetab \leq u_2$ is all of $[0,1]$. This is accomplished by showing that the set of $\eta$ for which we have $\{\vetab >0\}\cap D \Subset \{u_2>0\}\cap D$ is both open and closed; by construction the set is non-empty. The set is closed since the quantities involved vary continuously. We show it is open by supposing that there is an $\eta$ for which the free boundaries touch, that is $\vetab (x_0,t_0) =u_2(x_0,t_0)=0$.

 All points are regular from the right for $\vetab$ by properties of the sup-convolution. Since $\vetab$ touches $u_2$ at $(x_0,t_0)$, this point will be right regular for $u_2$. Additionally, by the assumption that $\vetab (x_0,t_0) =u_2(x_0,t_0)=0$, we have that $\veta (x_0,t_0)=0$ as well since $w$ vanishes where $u_2$ does. This means that the corresponding point $(y_0,s_0)$ on the free boundary of $u_1$ is left regular. Therefore, appealing to the asymptotic development results in \cite{B} we have
 \begin{align*}
 u_1 &\geq a_1^+\langle y-y_0, \nu_1 \rangle -a_1^-\langle y-y_0,\nu_1 \rangle +o(|y-y_0|)\\
 &\text{with } G(a_1^+,a_1^-) \geq 1 \text{ and equality along } t=-\gamma \langle y-y_0,\nu_1\rangle \gamma>0\\
 u_2 &\leq a_2^+\langle x-x_0, \nu_2 \rangle -a_2^-\langle x-x_0,\nu_2 \rangle +o(|x-x_0|) \\
&\text{with }G(a_2^+,a_2^-) \leq 1 \text{ and equality along }t=-\gamma \langle x-x_0,\nu_1\rangle, \gamma>0\\
 \veta &\geq a^+\langle x-x_0, \nu^* \rangle -a^-\langle x-x_0,\nu^* \rangle +o(|x-x_0|)\\
 \end{align*}
  where $\nu_1 =\frac{y_0-x_0}{|y_0-x_0|}$, $a^{\pm} =a_1^{\pm}|\tau|$, $\nu^* =\frac{\tau}{|\tau|}$ with
 \[
 \tau = \nu_1 +\frac{\varepsilon^2\phi_{\sigma\eta}(x_0,t_0)}{|y_0-x_0|}\nabla_x \phi(x_0,t_0)\mathbf{.}
 \]

 Now by the boundary Harnack comparison theorem we have $\frac{w}{u_2} \sim c$, so $w$ has the asymptotic development $ca_2^+$. Hence for $\vetab$ we have
 \[
 \vetab \geq \bar{a}^+\langle x-x_0, \nu^* \rangle -a^-\langle x-x_0,\nu^* \rangle +o(|x-x_0|)\mathbf{,}\\
 \] where $\bar{a}^+=a^++c\sigma \varepsilon a_2^+$.
Now recall that $G$ is Lipschitz continuous in both variables with Lipschitz constant $L_G$, increasing in the first, decreasing in the second. Moreover in [CS 9.14] it is shown that
 \begin{align*}
 |a_1^\pm -a^\pm| &\leq c(D_t\varepsilon\phi_{\sigma\eta}(x_0,t_0) +|\nabla \varepsilon\phi_{\sigma\eta}(x_0,t_0)|)\\ &\leq c(C\sigma\eta h \varepsilon+C\sigma\eta h \varepsilon+C\omega\varepsilon) \\
 &\leq\bar{c}\sigma h \varepsilon\mathbf{,}
 \end{align*}
the last inequality coming from the construction of the $\phi$. Specifically, we use the fact that $\eta \leq 1$ and $\omega \leq \varepsilon^2$ so the $C\omega\varepsilon$ term can be majorized by the linear term for small $\varepsilon$.  Now $\bar{a}^+ \geq a_1^+-\bar{c}\sigma\varepsilon h+c\sigma\varepsilon a_2^+$ and $a^- \leq a_1^-+\bar{c}\sigma\varepsilon h$. Hence we have
 \begin{align*}
 G(\bar{a}^+,a^-) &\geq G(a_1^+-\bar{c}\sigma\varepsilon h+c\sigma\varepsilon a_2^+,a_1^-+\bar{c}\sigma\varepsilon h)\\
 &\geq G(a_1^+,a_1^-) +L_G[(-\bar{c}\sigma\varepsilon h+ c\sigma \varepsilon a_2^+)-\bar{c}\sigma\varepsilon h]\\
 &= G(a_1^+,a_1^-)+L_G\sigma\varepsilon(-2\bar{c}h+ca_2^+)\\
 &\geq 1+L_G\sigma\varepsilon(-2\bar{c}h+ca_2^+)\mathbf{,}
 \end{align*}
which implies that $G(a^+,a^-) > 1$ provided $h \leq\frac{ca_2^+}{4\bar{c}}$. Our non-degeneracy condition forces $a_2^+\geq c>0$, so taking $h =\frac{ca_2^+}{4\bar{c}}$ we will have $(-2\bar{c}h+ca_2^+)>0$ and thus $G(a^+,a^-) > 1$ as desired.

 We finish the proof by appealing to the Hopf Principle. The difference $u_2-\vetab$ is a positive $\mathcal{L}$-supersolution in $\{\vetab >0\}$ vanishing at the boundary point $(x_0,t_0)$. This implies that $a_2^- \leq a^-$ and by the Hopf Principle we have $a_2^+>\bar{a}^+$. The properties of $G$ then imply that
 \[
 1\geq G(a_2^+,a_2^-)> G(\bar{a}^+,a^-)\mathbf{,}
 \]
 which contradicts $G(\bar{a}^+,a^-)>1$ above.

  Now recalling the properties of the $\varphi_\eta$ above, particularly
 \[
 \varphi_\eta\geq 1+k\eta h \text{ in } B'_{1/2}\times\left(\frac{-T}{2},\frac{T}{2}\right)
 \]
we have for $\eta=1$ $\varphi_\sigma \geq 1+k\sigma h$ and thus
\[
v_{\varepsilon(1+k\sigma h)} (x,t)\leq u_2(x,t)
\]
in the region $B'_{1/2}\times\left(\frac{-T}{2},\frac{T}{2}\right)$.
\end{proof}

\section{Spacial Regularity}

\subsection{Outline of Proof}

In the constant coefficient case regularity follows from applying the interior gain, then the propagation lemma, then rescaling and repeating.

Preventing us from applying this classical argument in our case is the extra $\omega_0 \leq \varepsilon^2$ hypothesis of our propagation lemma. This restricts our choice of $\tau$ for which we can apply the propagation lemma with $u_1 =u(x-\tau)$ and $u_2=u(x)$. The $\tau$ cannot be `too short', since if it is allowed to be arbitrarily short it forces the oscillation $\omega_0=0$. This means that we can carry fully monotonicity using the propagation lemma only in the constant coefficient case. This forces us to use $\varepsilon$-monotonicity in our variable coefficient problem.

The reader will recall that a function $u$ is $\varepsilon_0$-monotone in a unit direction $\tau$ if
\[
u(x)\geq u(x-\varepsilon\tau) \quad \text{for } \varepsilon\geq \varepsilon_0\mathbf{.}
\]
Strict $\varepsilon$-monotonicity, which is of importance in this problem, is similar but quantifies the `gap' between the two points:
\[
u(x)- u(x-\varepsilon\tau)\geq c\varepsilon^\beta u(x) \quad \text{for } \varepsilon\geq \varepsilon_0, \text{ some } \beta>0\mathbf{.}
\]
  Clearly if $u$ is fully monotone in a direction, then it is also $\varepsilon_0$-monotone for any $\varepsilon_0$ we choose.

Finally, it will be convenient to work with an alternate definition of $\varepsilon$-monotonicity, which is essentially equivalent to the one above. We say that $u$ is $\varepsilon$-monotone in the cone of directions $\Gamma(\nu,\theta)$ with defect angle $\delta$ if for any $\tau \in \Gamma(\nu,\theta-\delta)$  with $|\tau| =\varepsilon$ we have
\[
\sup_{B_{\varepsilon\sin \delta}(p)} u(q-\tau) \leq u(p)\mathbf{.}
\]
 In this case, the requirement of the propagation lemma is seen to be $\omega \leq (|\tau|\sin \delta)^2$.

 Our method of proof modifies the classical proof by accommodating this $\varepsilon$-monotonicity. An outline of the steps involved is as follows: Interior gain (given by Corollary~\ref{cor:int_gain}) is propagated to the free boundary by Lemma~\ref{lm:propagation}, but only for $\varepsilon$-monotonicity. The solution is then rescaled and by giving up part of the gain from the first two steps we can assert that the rescaled solution is fully monotone in a smaller cone away from the free boundary (see Lemma~\ref{lm:ep_to_full_mono_space}). This is all that is required to repeat the interior gain argument and at this point we can iterate the result. Special attention must be paid to the effect rescaling has on $\varepsilon$-monotonicity, as well as the amount of cone loss that occurs when in passing from $\varepsilon$-monotonicity to full monotonicity.

\subsection{Spacial Cone Enlargement}

In the lemma below, $r$ and $\lambda$ are constants (less than 1) chosen small enough later. In particular, $\lambda$ will be chosen by the calculation in Corollary~\ref{main_cor}. We take $\varepsilon_k =\lambda^k\varepsilon_0$ and $C_{r^k }= B_{r^k/2}\times(\frac{-r^{2k}T}{2},\frac{r^{2k}T}{2})$. $Q_R$ will be the quadratic cylinder $B_R\times(-R^2,R^2)$.

\begin{lm} \label{space_cone_enlarge}
Let $u$ be a solution to our problem in $B'_2\times (-2,2)$, monotone in the directions $\Gamma^x(e_n,\theta_0)\cup\Gamma^t(\eta,\theta_t)$\textbf{,} with $\eta$ in the span of $e_n$ and $e_t$. Then $u$ is $r^k\varepsilon_k$\textbf{-}monotone in $C_{r^k}$ in an expanded spacial cone of directions $\Gamma^x(\nu_1,\theta^x_1) =\Gamma^x_1$ with  defect angle $\delta_1 \leq c\delta_0$, $c<1$.
\end{lm}
 \textbf{Remark:}  Notice that we are asserting improved $\varepsilon$-monotonicity in smaller and smaller regions $C_{r^k}$, in \textit{the same} expanded cone of directions $\Gamma^x_1$. Increasing the cone opening iteratively will come later.

\begin{proof} We rescale $u$:
\[
u_r =\frac{u(rx,r^2t)}{r}\mathbf{,}
\] the rescaling factor $r$ to be fixed later in the proof. The rescaled function will still possess the same spacial monotonicity cone as the original. Additionally, it solves an equation with the rescaled coefficients $a_{ij}(rx,r^2t)$. The oscillation of these coefficients in controlled by $cr^\alpha$, $\alpha$ being the H\"{o}lder exponent. We will assume $r$ is small enough so that Corollary~\ref{cor:int_gain} and the related results from Section 3 can be applied to $u_r$.

Consider now a spacial vector $\tau \in \Gamma^x(e_n,\theta-\delta_0)$, where $\delta_0$ is the defect angle of the space cone, $|\tau| =\varepsilon \ll \delta_0$, $\bar{\varepsilon}=|\tau|\sin\delta_0$. Set $u_1(x,t) =u_r((x,t)-\tau)$. Additionally, assume that the defect angle of the space-time cone is less than that of the space cone.

From the monotonicity cone we have
\[
\sup_{B_{\bar{\varepsilon}(x)}}u_1(y,t) \leq u_r(x,t) \quad \text{in } B_{1}\times(-1,1).
\]
Note that this sup is performed over a space ball. However, we may assume that the same sort of result holds over a space-time ball
\begin{equation}\label{eq:full_ball_sup}
\sup_{B_{\bar{\varepsilon}(x,t)}}u_1(y,s) \leq u_r(x,t) \quad \text{in } B_{1}\times(-1,1)
\end{equation} since the defect angle in space is larger than that in time.

From Corollary~\ref{cor:int_gain} we have that there exists an enlarged cone of spacial directions $\tilde{\Gamma}_x$ in $\Psi$, the neighborhood of $(x_0,0)$. Let $\bar{\tau}$ be a unit (spacial) direction in this expanded cone $\tilde{\Gamma}_x$; then since this enlarged cone contains the old one we can write this direction as $\bar{\tau}=\alpha\tau-\beta e_n$, $\beta\geq 0$, where $\tau$ is a unit vector in the old cone.

Since $\bar{\tau}$ is a direction in which $u$ is increasing  we have $D_{\bar{\tau}} u_r \geq 0$. Using the above, this implies that
\[
D_\tau u_r\geq \frac{\beta}{\alpha}D_n u_r.
\]
Now, if we delete a small neighborhood $\mathcal{N}$ of the contact line $\Gamma\cap\tilde{\Gamma}_x$ between the old and new cones, we can force $\frac{\beta}{\alpha} \geq c\delta_0$, with $c$ depending on the size of the neighborhood $\mathcal{N}$ (see [CS], Section 9.4). We then obtain that for $\tau \in \Gamma_x \setminus \mathcal{N}$ we have
\[
D_\tau  u_r\geq c\delta_0 D_n u_r.
\]

We now demonstrate that a similar inequality holds in this region if we allow the direction to have a small time component of order $\delta_0$. \\

Let $\lambda_1$ and $\lambda_2$ be positive constants such that $\lambda_1^2+\lambda_2^2=1$ and $|\lambda_2|\leq \frac{c\delta_0}{2\bar{c}}$ (here $\lambda_1>1/2)$ where $\bar{c}$ is such that $|D_tu|\leq \bar{c}D_n u_r$ (this inequality is a consequence of the monotonicity cone). We then have
\[
\lambda_1 D_\tau u_r+\lambda_2 D_t u_r \geq (\lambda_1c\delta_0-\frac{c\delta_0}{2})D_n u_r \geq \bar{c}\delta_0 D_n u_r.
\]

Now if $\bar{\tau}=\tau+\bar{\varepsilon}\varrho$, where $\varrho$ can be any $(n+1)-$dimensional unit vector, we have
\[
u_r((x,t)-\bar{\tau})-u_r(x,t) =-D_{\bar{\tau}}u_r(\tilde{x},\tilde{t})|\bar{\tau}| \leq -c\bar{\varepsilon}\delta_0 D_nu_r(\tilde{x},\tilde{t}) \leq -c\bar{\varepsilon}\delta_0 u_r(x_0,0)\mathbf{.}
\]
In the last inequality we have used that $|\bar{\tau}|\geq c\varepsilon\geq c\bar{\varepsilon}$, $D_n u_r \sim \frac{u_r}{d}$  and the Harnack Inequalities. Note that $u_r((x,t)-\bar{\tau}) = u_r((x,t)-\tau -\bar{\varepsilon}\varrho)$ so as $\varrho$ varies we obtain in this region
\[
v_{\bar{\varepsilon}} (x,t) =\sup_{B_{\bar{\varepsilon}} (x,t)} u_1 \leq u_r(x,t) -c\bar{\varepsilon}\delta_0 u_r(x_0,0).
\]

Now by standard arguments, as in Section 4, this gap implies that there exists a small $h$ such that in $\Psi$ (with a different constant $c$)
\[
u_r(x,t) -v_{(1+h\delta)\bar{\varepsilon}}(x,t) \geq c\bar{\varepsilon}\delta_0 u_r(x_0,0).
\]

At this point we must restrict ourselves $\varepsilon$-monotonicity so that the propagation lemma can be applied. Select $\tau$ with $|\tau| = \lambda\varepsilon_0 =\varepsilon_1$ and take $r$ small enough so that $cr^\alpha \leq \varepsilon_1^2$ holds.

Now $cr^\alpha$ is the oscillation of the coefficient matrix $A$. This choice enables us to apply our propagation lemma and obtain that in $B_{1/2}\times (\frac{-T}{2},\frac{T}{2})$ (here $T$ is the `height' of the cylinder $\Psi$)
\[
u_r(x,t) \geq v_{(1+ch\delta_0)\bar{\varepsilon}_1}(x,t).
\]
Therefore $u_r$ is $\varepsilon_1$-monotone in an enlarged cone $\Gamma(\nu_1,\theta_1)$ in the region $B_{1/2}\times(\frac{-T}{2},\frac{T}{2})$ with defect angle $\delta_1=\frac{\pi}{2}-\theta_1 \leq c\delta_0$ with $c<1$. Back-scaling we obtain $r\varepsilon_1$-monotonicity for $u$ in the appropriately rescaled domain $C_r = B_{r/2}\times(\frac{-r^2T}{2},\frac{r^2T}{2})$

Now we can repeat this argument for $\varepsilon_k =\lambda^k\varepsilon_0$ and $r^k$, $\lambda$ to be chosen later. Precisely, $u$ is fully monotone in the original cone, so $u_{r^k}$ will be fully monotone in the original cone as well. Hence, it is $\varepsilon_k$-monotone no matter what we choose $\lambda$ to be. Additionally, parabolic blowups decrease the defect angle of the space-time cone so we have that~\eqref{eq:full_ball_sup} will hold for any $r^k$.

 We can repeat the cone enlargement arguments away from the free boundary to enlarge the spacial monotonicity cone. Finally, we can use the propagation lemma to transfer a portion of this new cone to the free boundary provided that $\omega_k \leq (\varepsilon_k)^2$. Since $\omega_k =cr^{\alpha k} $, we require that at each step $cr^{\alpha k} \leq \lambda^{2k}\varepsilon_0^2$\textbf{,} which can be arranged by coupling the choice of $r$ and $\lambda$.

This proves that in $C_{r^k}$ $u$ is $r^k\varepsilon_k$-monotone in the new, larger, cone of directions $\Gamma^x(\nu_1,\theta_1)$ (we get the same enlarged cone in each case). Alternatively $u_{r^k}$ is $\varepsilon_k$- monotone in $B_{1/2}\times(-\frac{T}{2},\frac{T}{2})$ in this same cone.
\end{proof}

 We now turn to the task of iteratively increasing the monotonicity cone in the above result. To do this we will need  to know that our solutions are fully monotone away from the free boundary since our interior gain results rely on full monotonicity. This in turn requires a strict $\varepsilon$-monotonicity not present in the result above. A slight modification of the proof, however, will yield the desired result. We explicitly observe that the enlarged cone in Corollary \ref{e_cone_gap} below is not the same as in Lemma~\ref{space_cone_enlarge}. 
 
 In the sequel we will let $\gamma, \delta$ be positive constants such that
\[
0<\gamma=\frac{1-\delta}{2}, \quad 0<\beta<\min\left\lbrace\gamma, \frac{\alpha+\delta -1}{2}\right\rbrace
\]
Notice in particular that this choice implies $\alpha+\delta-1>0$ and $\delta<1$. This $\delta$ is not to be confused with the defect angles $\delta_k$. The $M$ in the corollary below is determined by Lemma~\ref{lm:ep_to_full_mono_space} below.
\begin{cor}\label{e_cone_gap}
Let $u$ be a solution to our problem, monotone in the directions $\Gamma^x(e_n,\theta_0)\cup\Gamma^t(\eta,\theta_t)$, with $\eta$ in the span of $e_n$ and $e_t$. Then $u$ is $r^k\varepsilon_k$-monotone in $C_{r^k}$ in an expanded spacial cone of directions $\Gamma^x(\nu_1,\theta^x_1) =\Gamma^x_1$ with  defect angle $\delta_1 \leq c\delta_0$, $c<1$.

Alternatively $u_{r^k}$ is $\varepsilon_k$-monotone in $B_{1/2}\times(-\frac{T}{2},\frac{T}{2})$ in this cone. Furthermore, there exists an M such that\textbf{,} $M\bar{\varepsilon}_k^\gamma$ away from the free boundary, we have strict $\varepsilon_k$-monotonicity in these directions in the following sense:
\begin{equation}\label{eq:strict_e_mono}
u_{r^k}(p) -u_{r^k}(p-\tau) \geq c\sigma\bar{\varepsilon}_k^{1-\gamma}u_{r^k}(p)\mathbf{.}
\end{equation}
\end{cor}

\begin{proof}
We pick up the proof of the above lemma at the point where the propagation lemma is applied, slightly changing notation with $\sigma =c\delta_0$ and $p$ and $q$ space-time points, so that
\[
\sup_{B_{(1+h\sigma)\bar{\varepsilon}}(p)}u(q-\tau) \leq u(p).
\]
Reducing the radius of the ball to $(1+\frac{h\sigma}{2})\bar{\varepsilon}$ we have
\[
\sup_{B_{(1+\frac{h}{2}\sigma)\bar{\varepsilon}}(p)}u(q-\tau) +\min_{B_{(1+h\sigma)\bar{\varepsilon}}(p)}|\nabla u||\frac{h\sigma\bar{\varepsilon}}{2}|\leq \sup_{B_{(1+h\sigma)\bar{\varepsilon}}(p)}u(q-\tau).
\]
Now assume that $p$ is located a distance $M\bar{\varepsilon}^\gamma$ away from the free boundary with $M$ a large constant to be fixed later (see Lemma 11). We have $|\nabla u| \sim \frac{u}{d}$, $d$ being the distance to the free boundary, hence the minimum of $|\nabla u|$ can be compared to the minimum of $\frac{u}{d}$.

By the Harnack inequalities $u$ is comparable to $u(p)$, while for the distance we have
\[
M\bar{\varepsilon}^\gamma -(1+h\sigma)\bar{\varepsilon} \leq d \leq M\bar{\varepsilon}^\gamma+(1+h\sigma)\bar{\varepsilon}\mathbf{.}
\]
Since $1+h\sigma$ is bounded, this implies that $d$ is also comparable to $M\bar{\varepsilon}^\gamma$ for $\bar{\varepsilon}$ sufficiently small.

In turn this implies that
\begin{align*}
\sup_{B_{(1+\frac{h}{2}\sigma)\bar{\varepsilon}}(p)}u(q-\tau) &\leq \sup_{B_{(1+h\sigma)\bar{\varepsilon}}(p)}u(q-\tau)-\min_{B_{(1+h\sigma)\bar{\varepsilon}}(p)}|\nabla u||\frac{h\sigma\bar{\varepsilon}}{2} \\
&\leq \sup_{B_{(1+h\sigma)\bar{\varepsilon}}(p)}u(q-\tau)-C\frac{u(p)}{M\bar{\varepsilon}^\gamma}h\sigma\bar{\varepsilon}\\
&\leq \sup_{B_{(1+h\sigma)\bar{\varepsilon}}(p)}u(q-\tau)-Cu(p)\sigma\bar{\varepsilon}^{1-\gamma}\\
&\leq u(p) -Cu(p)\sigma \bar{\varepsilon}^{1-\gamma}.
\end{align*}

Thus\textbf{,} by reducing slightly the cone of monotonicity\textbf{,} strict monotonicity is obtained away from the free boundary.
\end{proof}

\subsection{Results regarding $\varepsilon$-monotonicity}

As mentioned above, we will need to know that our solutions enjoy full monotonicity away from the free boundary. The above corollary is the first step in this process. The remaining results have been collected in this section.\\

Set
\[
Q_{\sqrt{\varepsilon M}}(x^*,t^*) = B'_{\sqrt{\varepsilon M}}(x^*)\times (-M\varepsilon +t^*,M\varepsilon+t^*) \subset \Omega^+(u)\cap\{d_{x,t}>(M\varepsilon)^\gamma\}
\] and let $\zeta$ be the solution to the Dirichlet Problem
\begin{align*}
\zeta_t &=\mathcal{L}_{p^*} \zeta \quad \text{in } Q_{\sqrt{\varepsilon M}}(p^*)\\
\zeta &= u \quad \text{on } \partial_pQ_{\sqrt{\varepsilon M}}(p^*)\\
\end{align*}
where $\mathcal{L}_{p*} =\sum a_{ij}(p^*)D_{ij}u$, a constant coefficient operator.

We state the following result from [FS1] regarding $\zeta$. Recall that $\alpha$ is the H\"{o}lder exponent of our coefficients. The value of $M$ is the lemma below is to be determined later Lemma~\ref{lm:ep_to_full_mono_space}.\\

\begin{lm} $\mathrm{(2.5 \:in \:[FS1])}$
Let $u$ be our caloric function and $\zeta$ as above. Let $\alpha,\gamma,\delta >0$ be as in [Lemma 2.4 in FS1]. Then for every point $p^*$ outside a $(M\varepsilon)^\gamma$-neighborhood of  the free boundary of $u$, and for every point $p \in Q_{\sqrt{\varepsilon M}}$ the following estimates hold:
\begin{align}
|u(p)-\zeta(p)| &\leq C(M\varepsilon)^{\frac{\alpha}{2}+\delta}u(p^*)\\
|D_nu(p) -D_n\zeta(p)| &\leq C(M\varepsilon)^{\frac{\alpha}{2}+\frac{\delta}{2}}D_nu(p^*)\\
|D_tu(p)-D_t\zeta(p)| &\leq C(M\varepsilon)^{\frac{\alpha+\delta-1}{2}}D_nu(p^*).
\end{align}
\end{lm}

This next lemma allows us to transfer $\varepsilon$-monotonicity from $u$ to $\zeta$ provided we have strict monotonicity with the correct power.
\begin{lm}\label{e-mono_transfer}
Let $u$ and $\zeta$ be as above and suppose that $u$ is strictly $\varepsilon$-monotone in a direction $\tau$ in the following sense:
\[
u(p)-u(p-\varepsilon\tau)\geq c\varepsilon^{1-\gamma} u(p).
\] Then if $\varepsilon$ is sufficiently small, $\zeta$ is $\varepsilon$ monotone in the direction $\tau$.
\end{lm}

\begin{proof}
Define
\[
u_1(p) =u(p)-u(p-\varepsilon\tau), \quad \zeta_1(p) = \zeta(p) -\zeta(p-\varepsilon\tau)
\]
Using the above estimate we have
\begin{align}\label{lemma10}
\zeta_1(p) &\geq u_(p)-C_0(M\varepsilon)^{\delta+\frac{\alpha}{2}}u(p)\\
&\geq c(\varepsilon)^{1-\gamma} u(p) -C_0(M\varepsilon)^{\delta+\frac{\alpha}{2}}u(p)\\
&\geq 0
\end{align} for $\varepsilon$ small enough\textbf{,} provided $1-\gamma < \frac{\alpha}{2}+\delta$. Since
\[
\delta+\frac{\alpha}{2}=\left(\frac{\delta}{2}+\frac{\alpha}{2}\right)+\frac{\delta}{2}>\frac{1}{2}+\frac{\delta}{2}=1-\gamma\mathbf{,}
\]
we have the desired inequality and the proof is complete.
\end{proof}

\textbf{Remark:} The strict $\varepsilon$-monotonicity our solutions will enjoy from the previous section is
\[
u(p)-u(p-\varepsilon\tau)\geq c\sigma\bar{\varepsilon}^{1-\gamma} u(p) \geq c\delta_0^{2-\gamma}\varepsilon^{1-y}u(p)\textbf{,}
\] where $\sigma = c\delta_0$ and in future iterations we will have $\sigma_k=c\delta_k =c\bar{c}^k\delta_0$. We will be interested in applying the above lemma to a solution which is $\varepsilon_k$-monotone, in which case the $\bar{\varepsilon}$ appearing in the above will be $\bar{\varepsilon}_k$. Now for a fixed value of $\sigma$, there exists an $\varepsilon_0$ such that $\zeta_1(p)\geq 0$ as in the proof above. This value of $\varepsilon_0$ will depend on the size of $\sigma$, which could be problematic since $\sigma=c\delta_0$, and the defect angle will go to zero in our iteration.

However, in our iteration we will eventually have $\sigma_k =c\delta_k =c\bar{c}^k\delta_0$ and $\bar{\varepsilon}_k$. By choosing $\bar{c}$ close to 1, we can ensure that the calculation~\eqref{lemma10} remains valid when applied with $\delta_k$ and $\varepsilon_k$ since the $C_0(M\varepsilon)^{\delta+\frac{\alpha}{2}}u(p)$ term will also be decreasing.

We have then
\begin{lm}\label{lm:ep_to_full_mono_space}
Let $u$ be as in Corollary~\ref{e_cone_gap}.
 Then, if $\varepsilon_k$ is small enough and $M$ is large enough, $(M\bar{\varepsilon}_k)^\gamma$ away from the free boundary, $u_{r^k}$ is fully monotone in the cone $\Gamma^x_1(\nu, \theta_1-c_0\varepsilon_k^\frac{\alpha+\delta}{2})$, $c_0 >1$.
\end{lm}

\begin{proof}From Corollary~\ref{e_cone_gap}
\[
u_{r^k}(p) -u_{r^k}(p-\tau) \geq c\sigma\bar{\varepsilon}_k^{1-\gamma}u_{r^k}(p)\mathbf{,}
\] for $\tau \in \Gamma^x_1(\nu,\theta_1),|\tau|=\varepsilon_k$, $M\bar{\varepsilon}_k^\gamma$ from the free boundary.

Applying Lemma~\ref{e-mono_transfer} we conclude that $\zeta$ is also $\varepsilon_k$-monotone in the cone $\Gamma^x_1$ away from the free boundary. Now $\zeta$ solves a constant coefficient parabolic equation which we may assume is the heat equation. From the proof of Lemma 13.23 in [CS] we conclude that $\zeta$ is fully monotone in the cone of directions $\Gamma^x_1(\nu,\theta_1-c\varepsilon)$.

Note that this cone of monotonicity implies $|\nabla \zeta|$ is controlled (in this region) by $D_n\zeta$. This in turn implies our estimate for $|D_nu-D_n\zeta|$ extends to an estimate of $|D_eu-D_e\zeta|$, where $e$ is any spacial vector.

Using this, we have for any direction $e \in \Gamma^x_1(\nu,\theta_1-c\varepsilon)$
\[
D_e u(p^*)\geq D_e\zeta(p^*) -c(M\varepsilon_k)^{\frac{\alpha+\delta}{2}}D_nu(p^*)\mathbf{.}
\]
It follows then that $u$ is fully monotone in the direction $\bar{e}=e+c(M\varepsilon_k)^{\frac{\alpha+\delta}{2}}\nu$.

Hence $u$ is fully monotone in the cone $\Gamma^x_1(\nu,\theta-c\varepsilon_k-c(M\varepsilon_k)^{\frac{\alpha+\delta}{2}})$. Now at this point in the proof the size of $M$ has already been determined by invoking Lemma 13.23 in [CS]. We may assume $\varepsilon_k$ is small enough that we can majorize the loss term and have full monotonicity in $\Gamma^x_1(\nu,\theta_1-c_0\varepsilon^{\frac{\alpha+\delta}{2}})$, with $c_0>1$.
\end{proof}

\textbf{Remark} Lemma 13.23 in [CS] is stated in a slightly different context from what we have here. Its proof requires a control $|\zeta_t|\leq c|\nabla \zeta|$ which holds since such an estimate holds for $u$ and since we have the estimates between the derivatives of $u$ and $\zeta$. More precisely, we have 

\begin{align*}
\vert D_t\zeta\vert  &\leq |D_t\zeta-D_tu|+|D_t\zeta| \\
&\leq c(M\varepsilon)^{\frac{\alpha+\delta-1}{2}}D_nu(p^*).\\
\end{align*}
On the other hand, we have
\begin{align*}
|D_nu(p^*)| &\leq |D_nu(p^*)-D_n\zeta|+|D_n\zeta|\\
&\leq c(M\varepsilon)^\frac{\alpha+\delta}{2}D_nu(p^*)+|D_n\zeta|.\\
\end{align*}
Or,
\[
(1-c(M\varepsilon)^\frac{\alpha+\delta}{2})D_nu(p^*) \leq |D_n\zeta|.
\]
So if we take $\varepsilon$ small enough, we have control of $D_nu(p^*)$ by $D_n\zeta$. Here we have not specified an argument for $D_n\zeta$ since the above estimate holds for any point in the neighborhood.

Taken together, these imply the control of $D_t\zeta$ by $D_n\zeta$ needed in the proof of Lemma 13.23 in [CS] (naturally control by $D_n$ implies control by the full gradient). \\

Lastly, we quote Lemma 2.4 from [FS1] for the space-time cone.
\begin{lm}\label{e-mono-time}(Lemma 2.4 in [FS1]).\\
Let $\alpha\leq 1$ be the H\"{o}lder exponent of the $a_{ij}$ and $\beta,\delta,\gamma$ as indicated above.
\indent Suppose $u\geq 0$ is monotone in the $e_n$ direction and
\[
u(p)-u(p-\varepsilon\tau) \geq c\varepsilon^{1-\gamma+\beta}u(p) =c\varepsilon^{\frac{\alpha+\delta}{2}-1}u(p)
\]
for $d_p <\eta/4$ [this is the distance to the free boundary] where $\tau =\beta_1e_n+\beta_2e_t$ with $\beta_1>0$ $|\beta_2 \neq 0$ and $\beta_1^2+\beta_2^2 =1$. Then if $M = M(n,L)$ is large enough and $\varepsilon$ is small enough, outside a $(M\varepsilon)^\gamma$-neighborhood of the free boundary we have
\[
D_{t_\varepsilon} \geq 0
\]
where $t_\varepsilon =\tau+ c(M\varepsilon)^{(\alpha+\delta-1)/2}e_n$.
\end{lm}

\subsection{Regularity of the Free Boundary in Space}
 By combining the results from our cone enlargement lemma and the $\varepsilon$-monotonicity section we have the following corollary suitable for iteration.

\begin{cor}\label{main_cor}
Let $u$ be a solution to our problem, monotone in the directions $\Gamma^x(e_n,\theta_0)\cup\Gamma^t(\eta,\theta_t)$ with $\eta$ in the span of $e_n$ and $e_t$. Then $u$ is $r^k\varepsilon_k$-monotone in $Q_{r^k}$ in an expanded spacial cone of directions $\Gamma^x(\nu_1,\theta^x_1) =\Gamma^x_1$ with  defect angle $\delta_1 \leq c\delta_0$, $c<1$.

Alternatively $u_{r^k}$ is $\varepsilon_k$-monotone in $Q_1$. Furthermore, there exists an M such that\textbf{,} $M\bar{\varepsilon}_k^\gamma$ away from the free boundary, we have have strict $\varepsilon_k$-monotonicity in these directions in the following sense:
\begin{equation}\label{eq:strict_e_mono}
u_{r^k}(p) -u_{r^k}(p-\tau) \geq c\sigma\bar{\varepsilon}_k^{1-\gamma}u_{r^k}(p).
\end{equation} Finally, in this region, at a distance greater than $M\bar{\varepsilon}_k^\gamma$ from the free boundary, $u_{r^k}$ is fully monotone in a cone of directions $\bar{\Gamma}_1^x(\nu,\bar{\theta}_1)$ with $\bar{\delta}_1\leq \bar{c}\delta_0\mathbf{.}$
\end{cor}

\begin{proof} This is Corollary~\ref{e_cone_gap} except for the last part about full monotonicity.

 By Corollary~\ref{e_cone_gap} we have $u_{r^k}$ $\varepsilon_k$-monotone in $Q_1$ in the cone of directions $\Gamma^x(\nu_1,\theta_1)$, with~\eqref{eq:strict_e_mono} holding for directions in this cone. From Lemma~\ref{lm:ep_to_full_mono_space} $u_{r^k}$ is therefore fully monotone in the cone
 \[
 \Gamma^x(\nu_1,\theta-c_0\varepsilon_k^{\frac{\alpha+\delta}{2}}) :=\bar{\Gamma}_1^x.
 \] For notational convenience we will write $B$ for the power $\frac{\alpha+\delta}{2}$. In terms of the spacial defect angles, we know that $\delta_1 \leq c\delta_0$ with $c<1$. Let $\bar{\delta}_1$ denote the defect angle of the cone $\bar{\Gamma}_1^x$. It is readily seen that the worst case scenario occurs with $\varepsilon_1$. In this case we have
\[
\bar{\delta}_1 =\delta_1+c_0\varepsilon_1^B \leq c\delta_0+c_0\varepsilon_1^B.
\] We desire to preserve the geometric decay of the defect angles so we want $\bar{\delta}_1 \leq \bar{c}\delta_0$ with $\bar{c}<1$. So what we must prove is that there is an appropriate choice of $\lambda$ in the definition of $\varepsilon_k =\lambda^k\varepsilon_0$ that makes this possible.

Now $\varepsilon_1 =\lambda\varepsilon_0$\textbf{,} so we need
\[
c_0\lambda^B\varepsilon_0^B \leq c'\delta_0\mathbf{,}
\]
 where $c'$ is chosen so that $c + c' =\bar{c}<1$ and $c>2c'$. Now our starting assumption is that $\varepsilon_0 \ll \delta_0$ (and thus we can also assume $\varepsilon_0^B \leq \delta_0$; note that $B>1/2$) so it suffices to chose $\lambda$ so that
\[
\lambda^B \leq \frac{c'}{c_0}\mathbf{.}
\]
This would suffice to give $\bar{\delta}_1\leq \bar{c}\delta_0$.

\end{proof}

The calculation in the above proof will be of interest to us when we iterate. In particular, we need to ensure that the choice of $\lambda$ made in Corollary~\ref{main_cor} will also work in the iteration, where we need $\delta_k =\bar{c}^k\delta_0$. We take care of this now with the following result about cones.

\begin{lm}\label{calculation}
$\Gamma_k =\Gamma(\nu_k,\theta_k)$ is a sequence of cones, $\Gamma_k \subset \Gamma_{k+1}$ with defect angle $\delta_k \leq c^k\delta_0$. Let $\bar{\Gamma}_k =\Gamma_k(\nu_k,\theta_k -c_0\varepsilon_k^B)$, $B=\frac{\alpha+\delta}{2}$, with $\varepsilon_k =\lambda^k\varepsilon_0$, and defect angle $\bar{\delta}_k$. Then there exists a $\bar{c}<1$ such that $\bar{\delta}_k\leq \bar{c}^k\delta_0$.
\end{lm}
\begin{proof}

We have
\begin{align*}
\bar{\delta}_k &=\delta_k+c_0(\varepsilon_k)^B\\
&\leq c\bar{\delta}_{k-1}+ c_0(\lambda^k\varepsilon_0)^B\\
&\leq (c\bar{c}^{k-1} +(c')^k)\delta_0.
\end{align*}
 We want this last term to be less than $\bar{c}^k\delta_0$ for a choice of $\bar{c}$ independent of $k$. From $\bar{c}=c+c'$ (referring to the constants in the proof of Corollary~\ref{main_cor} above) and the binomial theorem we have
 \begin{align*}
(c\bar{c}^{k-1} +(c')^k) &= c\sum_{n=0}^{k-1} \binom {k-1}{n} c^{k-1-n}(c')^n +(c')^k\\
& =\sum_{n=0}^{k-1} \binom {k-1}{n} c^{k-n}(c')^n +(c')^k.
 \end{align*}
Whereas
\[
(c+c')^k=\sum_{n=0}^{k} \binom {k}{n} c^{k-n}(c')^n .
\]
Consider the term $n=1$. We have that the first expression has the term $(k-1)c^{k-1}c'$ while the second has $kc^{k-1}c' $. Provided $c'<c$ we will have
\begin{align*}
(k-1)c^{k-1}c'  +(c')^k <kc^{k-1}c' <(k-1)c^{k-1}c'+c^{k-1}c' =kc^{k-1}c'
\end{align*}

from which it follows that
\[
(c\bar{c}^{k-1} +(c')^k) <(c+c')^k.
\]
Therefore, with $\bar{c}=c+c'<1$, we will have $\bar{\delta}_k \leq \bar{c}^k\delta_0$ for any $k$.
\end{proof}
\textbf{Remark:} The results of this section imply that if a solution $v$ to our problem is strictly $\varepsilon$-monotone in the sense of~\eqref{eq:strict_e_mono} in a cone of directions $\Gamma_1$ with defect angle $\delta_1\leq c\delta_0$, then there exists $\bar{\Gamma}_1$, $\Gamma_0 \subset \bar{\Gamma}_1 \subset \Gamma_1$\textbf{,} with $v$ fully monotone away from the free boundary in $\bar{\Gamma}_1$, still preserving a decay of the defect angle $\bar{\delta}_1\leq \bar{c}\delta_0$.

\subsection{Final Spacial Iteration}
We reach the main result of this section.
\begin{cor}\label{cor: spacial_regularity}
The free boundary is a $C^{1,\alpha}$ surface in space.
\end{cor}
\begin{proof}
Combining Corollary~\ref{e_cone_gap} and the results in the $\varepsilon$-monotonicity section, we have that $u_{r^k}$ is $\varepsilon_k$-monotone in $B_{1/2}\times(-\frac{T}{2},\frac{T}{2})$ in the cone of directions $\bar{\Gamma}_1$ with $\bar{\delta}_1\leq \bar{c}\delta_0$. Furthermore, we have in this region
\[
\sup_{B_{\bar{\varepsilon_k}}(p)}u(q-\tau)\leq u(p)
\] for $\tau \in \bar{\Gamma}_1$, $|\tau|=\varepsilon_k$. Additionally, $u_{r^k}$ will be fully monotone in this cone of directions away from the free boundary.
We may assume that $T<1$. Then the quadratic cylinder $Q_{T/2} \subset B_{1/2}\times (-\frac{T}{2},\frac{T}{2})$.

This implies that $u_{r^2T/2}$ is $\frac{2}{T}\varepsilon_2$-monotone in $\bar{\Gamma}_1$ in the above $\sup$ sense, in all of $Q_1$ and is a solution in the larger region $Q_2$ (a technicality needed for the propagation lemma). Furthermore, $u_{r^kT/2}$ is fully monotone in $\bar{\Gamma}_1$ away from the free boundary in the region $\Psi$ by virtue of the results in Section 6.3.

We can then apply the proof of Lemma~\ref{space_cone_enlarge} and Corollary 4 to $u_{r^2T/2}$, concluding that in $B_{1/2}\times(-\frac{T}{2},\frac{T}{2})$, $u_{r^2T/2}$ is $\frac{2}{T}\varepsilon_2$-monotone in an enlarged cone of directions $\bar{\Gamma}_2(\nu_2,\bar{\theta}_2)$ with $\bar{\delta}_2\leq \bar{c}^2\delta_0$. As was the case for Lemma~\ref{space_cone_enlarge}, we have the same conclusion for $u_{r^kT/2}$, $k\geq 2$. Note that this region contains $B_{T/2}\times(-\frac{T}{2},\frac{T}{2})$.

 Using this observation, after back-scaling $u_{r^2T/2}$ we deduce that $u$ is $r^2\varepsilon_2$-monotone in a cone of directions $\bar{\Gamma}_2$ in the region $B_{r^2T^2/2^2}\times(-\frac{r^4T^3}{2^3},\frac{r^4T^3}{2^3}) \supset Q_{r^2T^2/2^2}$.

 In this way we construct a sequence of parabolic neighborhoods of the origin $Q_{r^kT^k/2^k}$ in which $u$ is $r^k\varepsilon_k$-monotone in a cone of directions $\bar{\Gamma}_k$. This implies that the free boundary of $u$ intersected with the time level $\{t=0\}$ is a $C^{1,\alpha}$ surface in space due to the following calculus lemma.

\end{proof}

 \begin{lm}\label{calculus lemma}
 Let $f$ be a function  defined in a region $\mathcal{D}$ monotone in contracting cylinders $C_k =B'_{R^k} \times (-b_k,b_k)$ in cones $\Gamma_k$, with defect angles $\delta_k \leq \lambda^k \delta_0$, $\lambda<1$. Additionally, assume $f(0)=0$. Then $\{f=0\}$ is a $C^{1,\alpha}$ surface.
 \end{lm}
\begin{proof} Since we can center the neighborhoods at any point on the free boundary, we have at once that each point on the free boundary possesses a genuine normal vector. It remains then only to show that these normal vectors vary with a modulus of continuity. It suffices for our purposes to assume the origin is one of the points, the other will be denoted by $x$; their corresponding normal vectors will be denoted by $\nu_x$ and $\nu_0$.

Select $k$ such that $R^{k+1}<|x|\leq R^k$ and let $\Gamma_{k+1}$ and $\Gamma_k$ be the corresponding monotonicity cones. Now the crucial observation is that the monotonicity cone is the same for any point in the corresponding region. In particular, both $x$ and $0$ have monotonicity cone $\Gamma_k$ since they are both in the region $C_k$. In turn this implies that the distance between the normal vectors $\nu_0$ and $\nu_x$ is controlled by the defect angle of the monotonicity cone:
\[
|\nu_x-\nu_0| \leq 2\delta_k.
\]
Now select $\alpha \in (0,1)$ such that $R^\alpha =\lambda$. Then we have
\begin{align*}
|\nu_x-\nu_0| \leq 2\delta_k &=c\lambda^k \\
&=c(R^\alpha)^k \\
&=c\left(\frac{R^{k+1}}{R}\right)^\alpha \\
&\leq C|x|^\alpha.
\end{align*}
\end{proof}

\section{Regularity of the Free Boundary in Space-Time}

We will now use similar ideas to prove that the free boundary has a space-time normal at every point which varies with a H\"{o}lder modulus of continuity. When taken together with the spacial regularity proved in the previous section this result will complete the regularity of the free boundary.

Having proved spacial regularity in the previous section we can orient our problem so that $e_n$ is the spacial normal at the origin. We will prove that there exists a space-time normal at the origin in the $e_n-e_t$ plane. This will be the full normal to the free boundary at that point. \\

The same technique used to prove the spacial regularity will be used for the space-time regularity. A technical difficulty arises early when following this line of argument however. Recall that in the spacial case we used the `sup-convolution' concept to describe the monotonicity cone. Precisely, given any $\tau \in \Gamma (e_n,\theta -\delta)$, $\delta$ the spacial defect angle, we have that
\[
\sup_{B'_\ep(x)} u(y-\tau,t) \leq u(x,t).
\]
Here the $B'$ denotes the ball in space only. We have that $\ep =|\tau|\sin \delta$. However, we need to know that the same statement holds over a full space-time ball $B_\ep$. Since parabolic rescalings depress the space-time defect angle $\mu$ we can assume that $\delta \geq \mu$ at every step in the iteration. This guarantees that the full ball $B_\ep$ is contained in the monotonicity cone.

In the present case however, the fact that the space-time defect angle $\mu$ is always smaller than the spacial defect angle $\delta$ poses a difficulty. It is still true that for $\tau \in \Gamma_t(\nu,\theta_t-\mu)$ we can take the $\sup$ statement over the `thin' ball, this time in the $e_n -e_t$ plane, but it is no longer true that we can take the sup over the full ball of radius $|\tau|\sin \mu$.

We have the following technical geometric lemma to address this problem.
\begin{lm}
Let $u$ be monotone in the directions $\Gamma_x(\theta_x)\cup \Gamma_t(\theta_t)$ with defect angles $\delta\geq \mu$ and $\delta\leq \pi/6$. Then there exists a $\kappa>0$ and a $\mu_0$ such that for any $\tau \in \Gamma_t(\theta_t-\kappa\mu)$ with $\mu\leq \mu_0$ the full ball $B_\ep$ centered at the endpoint of $\tau$ is contained in the monotonicity cone, with $\ep =|\tau|\sin \kappa\mu$.
\end{lm}

\begin{proof}
The space time cone is two dimensional in the plane $e_n-e_t$\textbf{,} while the spacial cone is a right cone in space. It therefore suffices to prove the result in three dimensions. Additionally, owing to the purely geometric nature of the lemma we may assume that the cones open along the positive $z$-axis. We will assume that the space-time cone opens along the $y$-axis.

Under these assumptions the elliptic monotonicity cone with defect angles $\delta$ and $\mu$ has parametric equations in the variables $s,t$
\[
(\cot(\delta) s \cos t, \cot(\mu) s \sin t, s)
\]
with $0\leq t\leq 2\pi$, $s\geq 0$. The vector
\[
v=(0, \cos \mu, \sin \mu)
\]
is on this cone; it is one edge of space-time direction with unit length.

 Define for $\mu^* =(1+\kappa) \mu =c\mu$ the unit vector $\tau$ as
\[
\tau =(0,\cos \mu^*, \sin \mu^*)\mathbf{.}
\]
Then $\tau$ will be an edge of the smaller space-time cone $\Gamma(\theta_t-\kappa\mu)$.

We want to show that for some choice of $\kappa$ no vector on the edge of the elliptic cone can make an angle with $\tau$ which is smaller than the angle $\tau$ makes with $v$. This would mean that the right cone of directions with axis $\tau$ and opening given by the angle between $v$ and $\tau$, which is $\kappa\mu$, would fit completely inside the elliptic cone. In turn this would imply that the ball with radius $|\tau|\sin \kappa\mu$ centered at the endpoint of $\tau$ is entirely contained in the elliptic cone, and this is the conclusion of the lemma.

 Let $\alpha (\;,\;)$ denote the angle between two vectors and let
\[
w = (\cot(\delta) x, \cot(\mu) y, 1)
\]
with $x^2+y^2 =1$; any vector on the outer edge of the cone lies in the same direction as $w$ for some such choice of $x,y$.

We want to show that for any such $w$
\[
\alpha(\tau, v) \leq \alpha(w,\tau)\mathbf{.}
\]
Or, since the cosine is decreasing in the first quadrant
\[
\cos(\alpha(\tau, v)) \geq \cos( \alpha(w,\tau))\mathbf{.}
\]
We then use the characterization of the dot product to obtain ($\tau$ and $v$ are unit vectors)
\[
\tau \cdot v \geq \frac{\tau \cdot w}{|w|}\mathbf{.}
\]
We compute (using $x^2+y^2=1$)
\[
\cos \mu^* \cos \mu +\sin \mu \sin \mu^* \geq \frac{y\cos \mu^* \cot \mu + \sin \mu^*}{\sqrt{1+(1-y^2) \cot^2\delta +y^2\cot^2\mu}}\mathbf{.}
\]
Now it can be shown directly from calculus that the right hand side as a function of $y$ with all other variables fixed increases to a maximum at
\[
y_M = \frac{\cot\mu\cos\mu^*(1+\cot^2\delta)}{\sin\mu^*(\cot^2\mu-\cot^2\delta)}\mathbf{.}
\]

When $y=1$ the two sides are equal so to obtain our desired inequality for all $0\leq y\leq 1$ it is necessary for $y_M \geq 1$. Recall that $\mu^* =c\mu$, $c>1$. We begin to estimate (we are assuming $\mu <\delta$ so $\cot \mu >\cot \delta$)
\begin{align*}
y_M &= \frac{\cot\mu\cos\mu^*(1+\cot^2\delta)}{\sin\mu^*(\cot^2\mu-\cot^2\delta)}\\
&\geq \frac{\cot\mu\cot\mu^*(1+\cot^2\delta)}{(\cot^2\mu)}\\
&=(1+\cot^2\delta)\frac{\cot\mu^*}{\cot\mu}\\
&= (1+\cot^2\delta)\frac{\tan\mu}{\tan\mu^*} = (1+\cot^2\delta)\frac{\tan\mu}{\tan c\mu}\mathbf{.}
\end{align*}

 Letting $\mu\to 0$ in this last line we obtain by L'H\^{o}pital's rule
\[
(1+\cot^2\delta) \frac{1}{c}\mathbf{.}
\]

We need
\[
(1+\cot^2 \delta)\geq c >1.
\]
Or, to provide some room
\[
(1+\cot^2 \delta)\geq c >2.
\]
Now, by assumption, $\delta\geq \pi/6$, so that $(1 +\cot^2\delta)>2$. Thus, we can find a $\mu_0$ such that for $\mu\leq \mu_0$ there exists a $c$ for which
\[
y_M \geq (1+\cot^2\delta)\frac{\tan\mu}{\tan c\mu} \geq 2\mathbf{.}
\]

In turn this implies that the ball centered at the tip of $\tau \in \Gamma_t(\theta_t-\kappa\mu)$ of radius $|\tau|\sin \kappa \mu$ will be completely contained in the monotonicity cone.
\end{proof}

Since we have already proved the spacial regularity of the solution and we know that parabolic rescalings depress the space-time defect angle $\mu$ we will assume throughout the rest of this section that the hypotheses of this lemma are satisfied.

Our proof of the space-time regularity can now proceed along the same lines as the spacial regularity. Namely, we prove an enlargement of the monotonicity cone away from the free boundary, transfer a portion of this enlarged cone to the free boundary and iterate via a parabolic rescaling. As in the spacial case care must be taken with the iteration necessary to accommodate working with $\ep$-monotonicity rather than full monotonicity.

A technical complication not present in the spacial case is the fact that parabolic rescalings enlarge the space-time cone $\Gamma_t$. Indeed, if $u$ has monotonicity cone $\Gamma_t$ with defect angle $\mu$ then $u_r$, a parabolic rescaling of $u$, will have a monotonicity cone which has defect angle $r\mu$. 

We note that this gain from rescaling is of no use to us in proving the enlargement of the monotonicity cone since any gain from the rescaling must be given back when the solution is back-scaled. We will write $^*{\Gamma}^t_0$ to mean the original cone $\Gamma^t_0$ dilated by the rescaling. The rescaling factor involved in this dilation will be clear from the context so we will suppress it from the already dense notation.\\

We begin with a lemma which is the analogy of Lemma~\ref{space_cone_enlarge} and Corollary~\ref{e_cone_gap} for the space-time case and which largely follows the same lines for its proof. The only major difference is the need to keep track of the dilation of the space-time cone due to the rescaling. As before, $M$ and $\ep_k =\lambda^k\ep_0$ for $\lambda<1$ are chosen later. Additionally, although the statement of the lemma is similar to the spacial case we do not necessarily have that the constants, including $r$ and $\lambda$, are the same.

\begin{lm}
Let $u$ be a solution to our problem, monotone in the directions $\Gamma^x(e_n,\theta_0)\cup\Gamma^t(\eta,\theta_t)$ with $\eta$ in the span of $e_n$ and $e_t$. Then $u$ is $r^k\varepsilon_k$-monotone in $C_{r^k}$ in an expanded space-time cone of directions $\Gamma^t(\nu_1,\theta^t_1) =\Gamma^t_1$ with  defect angle $\mu_1 \leq c\mu_0$, $c<1$.

Alternatively $u_{r^k}$ is $\varepsilon_k$-monotone in $B_{1/2}\times(-\frac{T}{2},\frac{T}{2})$ in the corresponding $r^k$-dilated cones $^*\Gamma_1^t$. Furthermore, there exists an M such that\textbf{,} $M\bar{\varepsilon}_k^\gamma$ away from the free boundary, we have have strict $\varepsilon_k$-monotonicity in these directions in the following sense:
\begin{equation}\label{eq:strict_e_mono_t}
u_{r^k}(x,t) -u_{r^k}((x,t)-\tau) \geq c\sigma\bar{\varepsilon}_k^{1-\gamma}u_{r^k}(x,t)\mathbf{.}
\end{equation}
\end{lm}
\begin{proof}

As in the spacial case we begin with a rescaling; as in that proof the choice of $r$ will be coupled to $\lambda$ and chosen later.
\[
u_r(x,t) =\frac{u(rx,r^2t)}{r}.
\]

For $u_r$ its space-time cone $^*\Gamma^t_0$ is described as the cone in the the $e_n- e_t$ plane with edges $e_t+Be_n$, $-e_t-Ae_n$. From Corollary~\ref{cor:int_gain} we have that the space-time cone enlarges. As in the spacial case we have either (the situation is simpler in this case since the cone is only two-dimensional)
\[
D_t u_r + BD_n u_r \geq cr\mu D_n u_r \quad \forall(x,t) \in \Psi
\]

or
\[
-D_t u_r -AD_n u_r \geq cr\mu D_n u_r \quad \forall(x,t) \in  \Psi\mathbf{.}
\]
 Recall that $^*\Gamma^t_0$ has defect angle $r\mu$; that is why an $r\mu$ appears on the right hand side.

We will assume that the first holds; the other case it treated similarly. For convenience we will denote by $\sigma$ the direction $e_t+Be_n$. Let $\tau$ be the direction in the $e_n-e_t$ plane which
lies below $\sigma$ by the angle $\kappa r\mu$.

Then define
\[
u_1(x,t) = u_r((x,t) - \tau).
\]
Then we have that
\[
\sup_{B_\ep(x,t)}u_1(x,t)\leq u_r(x,t)
\] throughout the whole cylinder by our previous lemma where $\ep = |\tau|\sin\kappa r\mu$.

Similar to the spacial case, the inequality
\[
D_\tau u_r \geq cr\mu D_n u_r \quad \forall(x,t) \in \Psi
\]
then holds. We needed the result about $\sigma$ to know which direction the cone was increasing in; once we have this information we only need to work with $\tau$. 

Next, we show that a similar inequaltiy holds for small perturbations of this direction $\tau$ by other directions.

Let $\lambda_1 >1/2$ and $\lambda_2$ be such that $\lambda_1^2 +\lambda_2^2 =1$.
Next recall that for any spacial direction $e$ we have $|D_e u_r| \leq c^* D_n u_r$ and a similar inequality holds for time derivatives by the monotonicity cone. Thus for $\varrho \in \mathbb{R}^{n+1}$
\[
\lambda_1 D_\tau u_r +\lambda_2 D_\varrho u_r \geq (\lambda_1c\mu - \lambda_2 c^*)D_n u_r \geq cr\mu D_n u_r
\]
provided $|\lambda_2|\leq \frac{cr\mu}{2c*}$.

 Set $\bar{\tau}=\tau
+\ep \varrho$ so that the above implies
\[
u_r((x,t) -\bar{\tau}) -u_r(x,t) = -D_{\bar{\tau}}u_r(\tilde{x},\tilde{t})|\bar{\tau}| \leq -c\varepsilon r \mu D_nu(\tilde{x},\tilde{t}) \leq -c\varepsilon r\mu u_r(x_0,0)\mathbf{.}
\]

As in the spacial case this inequality implies that
\[
u_r((x,t)-\bar{\tau})-u_r(x,t) = u_r((x,t)-\tau -\ep\varrho)-u_r(x,t) \leq -c\ep r\mu u_r(x_0,0).
\]
As $\varrho$ ranges over all possible direction we deduce that in the region $\Psi$
\[
v_\ep (x,t) := \sup_{B_\ep(x,t)} u_1(x,t) \leq u_r(x,t) - cr\mu u_r(x_0,0).
\]
This enlarges to yield that in $\Psi$ for a small $h$ we have
\[
u_r(x,t) -v_{(1+h\mu)\ep}(x,t) \geq c\ep r\mu u_r(x_0,0).
\]
which is the cone enlargement of the cone $^*\Gamma^t_0$ away from the free boundary.

It is at this point we must once again restrict ourselves to $\ep$-monotonicity so that the propagation lemma can be applied to $u_r$ and $u_1$. The remainder of the proof then proceeds in the same fashion as that of the spacial case.
\end{proof}

At this point the argument follows identical lines to that of the spacial case by using Lemma~\ref{e-mono-time} to produce a slightly smaller cone of direction $^*\bar{\Gamma}^t_1$ in which the solution is fully monotone away from the free boundary; additionally as in the spacial case a careful choice of $r$ and $\lambda$ results in the cones $^*\bar{\Gamma}_1^t$ still preserving the decay of the defect with $\mu_1 \leq \bar{c}\mu_0$. An iteration argument then implies the following corollary.

\begin{cor}\label{cor:space_time_iteration}
The solution $u$ is $r^k\varepsilon_k$ monotone in the parabolic neighborhoods of the origin $Q_{r^kT^k/2^k}$ in the cone of directions $\bar{\Gamma}^t_k$ which have defect angles $\mu_k \leq \bar{c}^k\mu$, $\bar{c}<1$.
\end{cor}

We arrive at the proof of our main theorem:
\begin{proof} \textit{(Theorem~\ref{thrm:main_thrm})}

The existence of a full normal at the origin follows by Corollary~\ref{cor:space_time_iteration} and the spacial regularity proved in Corollary~\ref{cor: spacial_regularity}. By centering this argument at different points we obtain that a normal vector to the free boundary exists at every point of the free boundary in $Q_{1/2}$. Furthermore, the spacial part of this normal vector varies with a H\"{o}lder modulus of continuity and the iteration from Corollary 7 implies that the space-time part of this normal also varies with a H\"{o}lder modulus of continuity.

 Together this implies both the existence of a normal vector $\eta(x,t)$ at each point on the free boundary and also that this normal vector varies with the moduli of continuity stated in Theorem~\ref{thrm:main_thrm}.

As in the proof of the main result in [ACS3], to finish we apply the results in [W] to our solution now that we know the free boundary is $C^{1,\alpha}$ for each time level. This implies that $\nabla_x u$ is continuous up to the boundary at every time level. Hence $u$ takes up its boundary condition with continuity and $u$ is a classical solution to our problem.
\end{proof}

\textbf{Acknowledgment:} The author would like to thank S. Salsa for his advice.

\end{document}